\documentclass{amsart}
\usepackage{amscd}
\usepackage{amssymb}
\usepackage[all, knot]{xy}
\usepackage{extpfeil}
\xyoption{all} \xyoption{arc}
\usepackage{hyperref}
\usepackage{color}
\usepackage{setspace}
\usepackage[hmargin=2.5cm,vmargin=2.5cm]{geometry}
\usepackage{MnSymbol}

\def\cW{\mathcal W}
\def\cE{\mathcal E}

\def\cF{\mathcal F}

\def\cG{\mathcal G}
\def\cP{\mathcal P}

\def\cL{\mathcal L}
\def\cM{\mathcal M}
\def\cN{\mathcal N}
\def\cQ{\mathcal Q}

\def\fm{\mathfrak m}

\def\voplus#1#2{\overset{#1}{\underset{#2}{\oplus}}}
\def\a{\alpha}
\def\b{\beta}

\def\fp{\mathfrak p}

\def\Hom{\operatorname{Hom}}

\def\cone{\operatorname{cone}}

\def\and{\, \text{ and } \,}

\newcommand{\xra}[1]{\xrightarrow{#1}}

\newcommand{\vp}[1]{\vspace{#1in}}
\newcommand{\hp}[1]{\hspace{#1in}}

\def\darrow#1#2{\xtofrom[#2]{#1}}

\numberwithin{equation}{section}
\theoremstyle{plain} 
\newtheorem{thm}[equation]{Theorem}

\newtheorem*{introthm*}{Theorem}

\newtheorem{cor}[equation]{Corollary}
\newtheorem{lem}[equation]{Lemma}
\newtheorem{prop}[equation]{Proposition}

\theoremstyle{definition}
\newtheorem{defn}[equation]{Definition}

\newtheorem{ex}[equation]{Example}

\theoremstyle{remark}
\newtheorem{rem}[equation]{Remark}

\doublespace

\begin{document}

\title{Chern character for matrix factorizations via Chern-Weil}

\author{Xuan Yu}

\maketitle

\begin{abstract}

Given a matrix factorization, we use the Atiyah class to give an algebraic Chern-Weil type construction to its Chern character; this allows us to realize the Chern character in an explicit way. It also generalizes the existing result to any smooth $k$-algebra $Q$ ($k$ a commutative ring containing $\mathbb{Q}$) and any $f\in Q$, which agrees with a recent result of Platt \cite{platt2012chern}. We also study some basic properties of the Chern character.

\end{abstract}

\section{Introduction}

Given a commutative ring $Q$ and an element $f\in Q$, a {\em matrix factorization} of $f$ is a finitely generated $\mathbb{Z}/2$-graded projective $Q$-module $M=M_0\oplus M_1$ together with an odd degree endomorphism $d_M$ such that $d_M^2=f\cdot 1_M$. Matrix factorizations were introduced by Eisenbud \cite{eisenbud1980homological} to study modules over the ring $Q/f$. 

The theory of matrix factorizations is very active in recent years and one of the things that's been heavily studied is the Chern character. Polishchuk-Vaintrob \cite{polishchuk2010chern} establishes a Chern character for matrix factorizations when $Q=k[[x_1, \cdots, x_n]]$ ($k$ is a field of characteristic $0$) and $f$ is an isolated singularity. Dyckerhoff-Murfet \cite{dyckerhoff2010kapustin} produce the same Chern character by an explicit description of a local duality isomorphism (also later in the frame work of pushing forward of matrix factorizations \cite{dyckerhoff2011pushing}). Carqueville-Murfet \cite{carqueville2012adjunctions} studies the bicategory of Landau-Ginzburg models and recover the Chern character in this setting. Recently, Platt \cite{platt2012chern} gives an explicit formula for the boundary bulk map; and in the case when the matrix factorization admits a connection, an explicit formula for the Chern character.

In this paper, we use the Atiyah class $At$ of matrix factorizations to give an algebraic Chern-Weil type construction to the Chern character. In order to do this, we need to require $At$ to be a strict morphism of matrix factorizations. Our observation is that $At$ is not a strict morphism; however, $id+At$ is. The Chern-Weil type construction allows us to generalize the Chern character to any finitely generated smooth $k$-algebra $Q$ ($k$ a commutative ring containing $\mathbb{Q}$) and any element $f$ of $Q$. Carqueville-Murfet \cite{carqueville2012adjunctions} defines and studies Atiyah classes in a more general setting, the one we use in this paper is the more elementary one that has already been given out in an earlier paper of Dyckerhoff-Murfet \cite{dyckerhoff2011pushing}. We also show that the Chern character is in fact a ring homomorphism from the Grothendieck ring of the homotopy category of matrix factorizations to its Hochschild homology (Corollary 5.18). As a byproduct, this construction gives a very concrete way of constructing a complex that can be used to calculate the Hochschild homology for matrix factorizations \cite{walker2013support}. We also prove the naive funtoriality (Proposition 5.21) at the end.

\smallskip

\smallskip

\emph{Convention} Throughout all rings are Noetherian, commutative and unital. All modules are finitely generated. All complexes are bounded.

\smallskip


\section{Matrix factorizations}

First we recall the theory of matrix factorizations. 

\begin{defn} Given a ring $Q$ and an element $f$ of $Q$, a {\em matrix factorization}  $\cM$ of $f\in Q$ is a
$\mathbb{Z}/2$-graded $Q$-module $M=M_0\oplus M_1$, where $M$ is a finitely generated
projective $Q$-module, together with an odd endomorphism $$ d = \begin{bmatrix} 0 & d_1 \\
d_0 & 0 \end{bmatrix}
$$ such that $d\circ d=f\cdot 1_M$.

\end{defn}

Equivalently, a matrix factorization for $(Q,f)$ consists of a pair
of finitely generated projective $Q$-modules $M_0$ and $M_1$ and $Q$-linear maps $d_0:
M_0 \rightarrow M_1$ and $d_1: M_1 \rightarrow M_0$ such that each
composition is multiplication by $f$: $$d_0\circ d_1=f\cdot 1_{M_1}  \, \text{ and } \,  d_1\circ d_0=f\cdot
1_{M_0}.$$

We write a matrix factorization as $$(M_1 \darrow{d_1}{d_0} M_0)
\text{ or } (M_1\xrightarrow{d_1} M_0\xrightarrow{d_0} M_1).$$ Note that the degree $1$ part is on the left for the first version. For the second version, we have the degree
$0$ piece in the middle and degree $1$ pieces elsewhere.

\begin{ex}

Let $Q=\mathbb{C}[[x]]$ and $f=x^n$, then we have matrix factorizations $(Q \darrow{x^i}{x^{n-i}} Q),$ for all $i$.

\end{ex}

\begin{ex} Given $Q=\mathbb{C}[[x,y,z]], f=xy+yz+zx$, then it's easy to check that $(Q^2 \darrow{d_1}{d_0} Q^2),$ with $d_1=\begin{bmatrix} z & y
\\x & -x-y \end{bmatrix}$ and $d_0=\begin{bmatrix} x+y & y \\x & -z
\end{bmatrix}$ is a matrix factorization of $f$. \end{ex}

\smallskip

\begin{defn}

A {\em strict morphism} of matrix factorizations from $\cM$ to $\cN$ is a $\mathbb{Z}/2$-graded $Q$-linear map of degree zero $\alpha: \cM\rightarrow \cN$ such that $d^{\cN} \circ \alpha = \alpha \circ d^{\cM}$

Equivalently, a strict morphism is a pair of $Q$-linear maps $\alpha_0:M_0\rightarrow N_0$ and $\alpha_1:M_1\rightarrow N_1$ such that the two evident squares commute.

\end{defn}

We write $MF(Q,f)$ for the category of all matrix factorizations of $(Q,f)$ with morphisms given by the set of strict morphisms. 

\begin{defn}
Two strict morphisms $\alpha, \beta: \cM \to \cN$ are {\em homotopic} if there exists morphisms $h_1 \in \Hom(M_1,N_0)$ and $h_0\in\Hom(M_0,N_1)$ such that $$d^{\cN} \circ h + h \circ d^{\cM} = \alpha - \beta.$$

\end{defn} 

We visual a homotopy as the following: \begin{displaymath}
    \xymatrixcolsep{3pc}\xymatrix{
        M_1 \ar[d]_{\alpha-\beta} \ar[r]^{d^{\cM}}  & M_0 \ar@{.>}[ld]^{h_0} \ar[d]^{\alpha-\beta} \ar[r]^{d^{\cM}}  & M_1  \ar@{.>}[ld]^{h_1} \ar[d]^{\alpha-\beta}  \\
        N_1 \ar[r]_{d^\cN} & N_0 \ar[r]_{d^\cN} & N_1 }
\end{displaymath}

This is an equivalent relation and is preserved by composition of strict morphism. The homotopy category $[MF(Q,f)]$ is obtained from $MF(Q,f)$ by modding out the hom sets by this equivalence relation.

A strict morphism $\alpha: \cM \to \cN$ that becomes an isomorphism in $[MF(Q,f)]$ is called a {\em homotopy equivalence}, i.e., $\alpha$ is a homotopy equivalence if and only if there exists a strict morphism $\beta: \cN \to \cM$ such that $\a \circ \b$ and $\b \circ \a$ are each homotopic to the appropriate identity map.

\begin{defn} 

For $\cM \in MF(Q,f)$, the {\em shift} $\cM[1] \in MF(Q,f) $ is defined to be: $$\left(M_1 \darrow{d_1}{d_0} M_0\right)[1] = \left(M_0\darrow{-d_0}{-d_1} M_1\right).$$

\end{defn}

\begin{defn} 

The {\em cone} of a strict morphism $\alpha: \cM \to \cN$ is the following matrix factorization of $(Q,f)$: $$\cone(\alpha) = \left( N_1\oplus M_0 \darrow{
\begin{bmatrix}
d^N_1 & \alpha_0 \\
0 & -d^M_0
\end{bmatrix}}
{
\begin{bmatrix}
d^N_0 & \alpha_1 \\
0 & -d^M_1
\end{bmatrix}
}  N_0\oplus M_1 \right).
$$ \end{defn}

There are canonical maps $$\cN \to \cone(\alpha) \and \cone(\alpha)\to \cM[1],$$ as the classic situation of chain complexes. These will give the ``distinguished triangles'' in the triangulated structure discussed in the next proposition.

\begin{prop} For any ring $Q$ and element $f \in Q$, the category $[MF(Q,f)]$ is a triangulated category. The shift functor is $M \mapsto M[1]$ and the distinguished triangles are those isomorphic (in $[MF(Q,f)]$) to triangles of the form $$\cM \xra{\alpha} \cN \xra{} \cone(\alpha) \xra{} \cM[1] $$ for any strict morphism $\alpha$.
\end{prop}

\begin{defn} Given $f, f'\in Q$ and matrix factorizations
$\cM \in MF(Q,f)$ and $\cN \in MF(Q,f')$, the {\em tensor product} of
$\cM$ and $\cN$ is $$\cM \otimes_{mf} \cN:=((M_1\otimes_Q N_0)\oplus(M_0\otimes_Q N_1)
\darrow{d_{M\otimes N}}{d_{M\otimes N}} (M_0\otimes_Q
N_0)\oplus(M_1\otimes_Q N_1)),$$  where $d_{M\otimes N}(m\otimes n)=d_M(m)\otimes n + (-1)^{|m|}m\otimes d_N(n)$ for simple, homogeneous tensor $m\otimes n$. The tensor product is a matrix factorization of $f+f'$. For further details, see \cite{yoshino1998tensor}.

\end{defn}

The tensor product $-\otimes_{mf}-$ of matrix factorizations is well-defined on the homotopy category of matrix factorizations.

\begin{prop}(Lemma 2.2 of \cite{yoshino1998tensor})

Given any three matrix factorizations $\cM,\cN$ and $\cL$, we have $$(\cM\oplus \cN)\otimes_{mf} \cL\cong (\cM\otimes_{mf} \cL)\oplus(\cN\otimes_{mf} \cL) $$

\end{prop}

For a complex of $Q$-modules, we have the following definition.

\begin{defn}

Given any complex $C^{\cdot}$ of $Q$-modules we denote by $C^{\cdot}_{\mathbb{Z}/2}$
the $\mathbb{Z}/2${\em -folding}, which has $\bigoplus\limits_{i \in
2\mathbb{Z}} C^i$ in degree zero and $\bigoplus\limits_{i \in
2\mathbb{Z}+1} C^i$ in degree one, together with the obvious
differentials. This is a matrix factorization of $0$.

\end{defn}

\begin{rem} We can talk about tensor product (in the sense of definition 2.9) between complexes of projective $Q$-modules and matrix factorizations. If one of the factors in the tensor product is simply a projective $Q$-module, or more generally a complex of projective $Q$-modules, we first view it as a matrix factorization of zero using the $\mathbb{Z}/2$-folding (for the case of a single module, we follow the usual convention by placing it in the degree $0$ piece of a complex), then tensor everything as matrix factorizations, i.e., $P^{\cdot}\otimes M:=P^{\cdot}_{\mathbb{Z}/2}\otimes_{mf} M$ for $P^{\cdot}$ a complex. We have the following proposition addressing the
problem of compatibility.

\end{rem}

\begin{prop}

Given two complexes $X^{\cdot}$ and $Y^{\cdot}$ of projective $Q$-modules, we have
$(X^{\cdot}\otimes_{cx}Y^{\cdot})_{\mathbb{Z}/2}= X^{\cdot}_{\mathbb{Z}/2}\otimes_{mf}
Y^{\cdot}_{\mathbb{Z}/2}$, where $\otimes_{cx}$ stands for the usual tensor product of complexes.

\end{prop}

\begin{proof} For note that the underlying modules for $(X^{\cdot}\otimes_{cx} Y^{\cdot})_{\mathbb{Z}/2}$ 
and $X^{\cdot}_{\mathbb{Z}/2}\otimes_{mf} Y^{\cdot}_{\mathbb{Z}/2}$ are identical.

Indeed, we have $$((X^{\cdot}\otimes_{cx}Y^{\cdot})_{\mathbb{Z}/2})_1=\bigoplus\limits_{k
\text{ is odd }}(\bigoplus\limits_{i+j=k}(X^i\otimes
Y^j))$$ and $$(X_{\mathbb{Z}/2}\otimes_{mf}
Y_{\mathbb{Z}/2})_1=[(\bigoplus\limits_{i \text{ is
odd}}X^i)\otimes(\bigoplus\limits_{j \text{ is
even}}Y^j)]\bigoplus[(\bigoplus\limits_{i \text{ is
even}}X^i)\otimes (\bigoplus\limits_{j \text{ is
odd}}Y^j)]=\bigoplus\limits_{k \text{ is odd
}}(\bigoplus\limits_{i+j=k}(X^i\otimes Y^j)).$$ Similarly, $((X^{\cdot}\otimes_{cx} Y^{\cdot})_{\mathbb{Z}/2})_0=(X^{\cdot}_{\mathbb{Z}/2}\otimes_{mf} Y^{\cdot}_{\mathbb{Z}/2})_0.$

The fact that the differentials are the same can be seen by carefully keeping track of where elements go.

\end{proof}


\section{Algebraic Chern-Weil Theory}

We review the basic Chern-Weil theory from the algebraic point of view in this section, which will be used later in our construction. From now on, $Q$ is a finitely generated commutative $k$-algebra, where $k$ is a commutative ring. All modules are finitely generated. Also, let $\Omega_{Q/k}^1$ be the $Q$-module of differential $1$-forms and $\Omega_{Q/k}^n:=\wedge_Q^n\Omega_{Q/k}^1$, the $Q$-module of differential $n$-forms. For details, see \cite{loday1998cyclic}.

\begin{defn}
Let $Q$ be a commutative $k$-algebra and $E$ a $Q$-module. A {\em connection} on the $Q$-module $E$ is a $k$-linear map $\nabla: E\rightarrow \Omega_{Q/k}^1\otimes_Q E$ such that for any $e\in E$ and $q\in Q$ the following Leibniz rule holds: $$\nabla(qe)=(dq)\otimes e +q\nabla(e).$$

\end{defn}

Just like the exterior differential operator $d$, a connection $\nabla$ can be extended canonically to a map, which we still denote by $\nabla$, $$ \Omega_{Q/k}^{\bullet}\otimes_Q E \rightarrow \Omega_{Q/k}^{\bullet+1}\otimes_Q E$$ such that for any homogeneous element $u\in\Omega_{Q/k}^{\bullet}$ and $e\in E$ $$\nabla(u\otimes e)=(du)\otimes e+(-1)^{|u|}u\wedge\nabla(e).$$

\begin{ex}

For $E=Q$, the exterior differential operator $d$ is a connection. More generally, if $E=Q^r$, $$\Omega_{Q/k}^{\bullet}\otimes_Q E\cong (\Omega_{Q/k}^{\bullet})^r \, \text{ and } \, d\cdot I_r: (\Omega_{Q/k}^{\bullet})^r \rightarrow (\Omega_{Q/k}^{\bullet+1})^r$$ is a connection for $Q^r$.

\end{ex}

Every finitely generated projective module $E$ possesses a
connection. Given such an $E$, choose an idempotent $e$ in $M_r(Q)$ for some $r$ such that $E=Im(e)$. Then,
from the connection on $Q^r$ in the previous example, we get a
connection on $E$ as the following composition:
$$ \xymatrix{\Omega_{Q/k}^\bullet\otimes E \ar@{^{(}->}[r] & \Omega_{Q/k}^\bullet\otimes Q^r \ar[r]^{d\cdot I_r} & \Omega_{Q/k}^{\bullet+1}\otimes Q^r \ar[r]^{1\otimes e} & \Omega_{Q/k}^{\bullet+1}\otimes E }$$

\begin{defn}

This connection on $E=Im(e)$ is called the {\em Levi-Civita} connection by analogy with the classical situation in differential geometry.

\end{defn}

\begin{defn}

The {\em curvature} $R$ of a connection $\nabla$ on a finitely generated $Q$-module $E$ is defined to be $$R:=\nabla\circ\nabla: E\rightarrow \Omega_{Q/k}^2\otimes_{Q} E.$$

\end{defn}

It can be shown that $R$ is $Q$-linear.

\begin{prop} (Example 4.2.6 of \cite{huybrechts2005complex})

\begin{enumerate}

\item Let $E_1$ and $E_2$ be projective Q-modules with connections $\nabla_1$ and $\nabla_2$, respectively. Then for $e_1\in E_1$ and $e_2\in E_2$, we set 
$$\nabla(e_1\oplus e_2)=\nabla_1(e_1)\oplus \nabla_2(e_2).$$

This defines a natural connection on the direct sum $E_1\oplus E_2$.

\item In order to define a connection on the tensor product $E_1\otimes E_2$, one defines $$\nabla(e_1\otimes e_2)=\nabla_1(e_1)\otimes e_2+e_1\otimes\nabla_2(e_2).$$ Note that the second component naturally lands in $E_1\otimes_Q(\Omega_{Q/k}^1\otimes_Q E_2)$, so we need to apply an isomorphism $\tau: E_1\otimes_Q\Omega_{Q/k}^1\cong \Omega_{Q/k}^1\otimes_Q E_1$, which sends $e_1\otimes w$ to $w\otimes e_1$ to make it into an element of the target module $\Omega_{Q/k}^1\otimes_QE_1\otimes_QE_2$.

\end{enumerate}

\end{prop}

Before getting into the next proposition, we want to inform the reader that by $exp(R)$ we mean the series $1+R+\frac{R^2}{2!}+\frac{R^3}{3!}+\cdots+\frac{R^n}{n!}+\cdots.\in \prod_n End_Q(E)\otimes_Q\Omega_{Q/k}^{2n}.$ In order to do this, we need to make the extra assumption that $k\supseteq \mathbb{Q}$. The exterior operator $d$ can be extended to maps $\Omega_{Q/k}^n\rightarrow\Omega_{Q/k}^{n+1}$ (for any $n\in \mathbb{N}$) by $$d(a_0da_1\cdots da_n)=da_0da_1\cdots da_n.$$ Since $d(1)=0$ it is immediate that $d^2=0$, and the following sequence $$\xymatrix{ Q=\Omega_{Q/k}^0 \ar[r]^d & \Omega_{Q/k}^1 \ar[r]^d & \cdots \ar[r]^d & \Omega_{Q/k}^n \ar[r]^d & \cdots} $$ is a complex called the \em{de Rham complex} of $Q$ over $k$. The cohomology groups of the de Rham complex are
denoted $H_{DR}^n(Q)$ and are called the \em{de Rham cohomology} of $Q$ over $k$.

\begin{prop}(Proposition 8.1.6 of \cite{loday1998cyclic}) The homogeneous component of degree 2n of $ch(E,\nabla):=tr(exp(R))$ is a cycle in $\Omega^{2n}_{Q/k}$ (of the de Rham complex), where $tr$ stands for the trace map for projective modules (details in Section ).  

\end{prop}

This proposition implies that $ch(E,\nabla)$ defines a cohomology class in the de Rham cohomology of $Q$.

\begin{thm} (Theorem-Definition 8.1.7 of \cite{loday1998cyclic})
The cohomology class of $ch(E,\nabla):=tr(exp(R))$ is independent of the connection $\nabla$ and defines an element $$ch(E)\in \prod\limits_{n\geqslant 0} H^{2n}_{DR}(Q)$$ which is called the Chern character of the finitely generated projective $Q$-module $E$.

\end{thm}

\begin{thm} (Theorem 8.2.4 of \cite{loday1998cyclic})

The Chern character induces a ring homomorphism $ch: K_0(Q) \rightarrow H^{even}_{DR}(Q).$

\end{thm}


\smallskip

\section{Main constructions}

Given a $k$-algebra $Q$, $k$ a Noetherian commutative ring, for a matrix factorization $\cE=(E_1\xrightarrow{A} E_0\xrightarrow{B} E_1)$ of $f\in Q$ (so the odd endomorphism of this matrix factorization is $d=\begin{bmatrix} 0 & A \\ B & 0 \end{bmatrix}$), choose connections $\nabla_i:E_i\rightarrow \Omega^1_{Q/k}\otimes_Q E_i$ for
$i=0,1$. By Proposition 3.5, $\nabla_0$ and $\nabla_1$ induce a natural connection for the underlying module $E=E_0\oplus E_1$ of the form $$ \nabla = \begin{bmatrix} \nabla_0 & 0 \\ 0 & \nabla_1 \end{bmatrix}.$$

\begin{defn} (\cite{dyckerhoff2011pushing}) The {\em Atiyah class} of $\mathcal{E}$, written $At_{\mathcal{E}, \nabla}$ (or simply
just $At$ if there is no confusion), is the map $$\nabla\circ d - (1\otimes d)\circ \nabla =^{def}At_{\mathcal{E}, \nabla}:  \mathcal{E}\rightarrow
\Omega^1[1]\otimes_{mf} \mathcal{E}.$$ As usual, we regard the single module $\Omega^1$ as a complex with $\Omega^1$ in the degree $0$ piece, so $\Omega^1[1]$ is the shift of this complex. Since we will take the $\mathbb{Z}/2$-folding while tensoring a complex with a matrix factorization, $\Omega^1[1]$ is really the matrix factorization $(\Omega^1\darrow{}{} 0)$ (degree $1$ piece on the left). See Remark 2.12 for details of the tensor product of a module and a matrix factorization.

\end{defn} 

It's easy to check that the Atiyah class is a $Q$-module homomorphism from the $\mathbb{Z}/2$-graded $Q$-module $E$ to the $\mathbb{Z}/2$-graded $Q$-module $\Omega^1\otimes_Q E$.

Compositions of Atiyah classes are defined in the following way. For example, by definition, we have $(1\otimes At)\circ At : \mathcal{E}\rightarrow \Omega^1[1] \otimes_{mf} \mathcal{E}\rightarrow \Omega^1[1]\otimes_{mf}\Omega^1[1]\otimes_{mf} \mathcal{E}$, which is defined as $$((1\otimes\nabla)\circ(1\otimes d) - (1\otimes 1\otimes d)\circ(1\otimes\nabla))\circ(\nabla\circ d - (1\otimes d)\circ\nabla).$$ For simplicity, we denote this composition by $\widetilde{At^2}$. Similarly, we can define $\widetilde{At^i}$ (for nature numbers $i\geqslant 2$) recursively by $$\widetilde{At^i}:=(1_{\underbrace{\Omega^1\otimes\cdots\Omega^1}_{i-1}}\otimes At)\circ \widetilde{At^{i-1}}.$$ Hence the
map $$\widetilde{At^i}: \mathcal{E}\rightarrow \overbrace{\Omega^1[1]\otimes_{mf}\cdots\otimes_{mf}\Omega^1[1]}^i\otimes_{mf} \mathcal{E}$$ has $i$ copies of $\Omega^1[1]$
in the target.

\begin{defn}: Define $At^i$ to be the composition: 

$$\mathcal{E} \xrightarrow{\widetilde{At^i}} \overbrace{\Omega^1[1]\otimes_{mf}\cdots\otimes_{mf}\Omega^1[1]}^i\otimes_{mf} \mathcal{E} \xrightarrow{\wedge} \Omega^i[i]\otimes_{mf} \mathcal{E}.$$

\end{defn}

Note that we have $At=At^1=\widetilde{At^1}$.

\smallskip

\subsection{Basic construction: the strict morphism $\varphi$}\hp{1}

\begin{defn}

Define $\mathcal{E}^{(1)} = (\xymatrix{ Q \ar[r]^{df\wedge} & \Omega^{1}})\otimes_{mf} \mathcal{E}$, with $Q$ in degree $0$ and $\Omega^1$ in degree $1$.

\end{defn}

Explicitly, $\mathcal{E}^{(1)}$ is by definition the following: $$\mathcal{E}^{(1)} = (\xymatrix{\voplus{E_1}{\Omega^{1}\otimes E_0}\ar[r]^{\overline{A}} & \voplus{E_0}{\Omega^{1}\otimes E_1} \ar[r]^{\overline{B}} & \voplus{E_1}{\Omega^{1}\otimes E_0} })$$ with $\overline{A} = \begin{bmatrix} A & 0 \\
df\wedge & -B \end{bmatrix} $ and $ \overline{B} = \begin{bmatrix} B & 0 \\ df\wedge & -A \end{bmatrix} $. For details, see Proposition 2.13. (Note that the $\mathbb{Z}/2$-folding of $(\xymatrix{ Q \ar[r]^{df\wedge} & \Omega^{1}})$ is the matrix factorization $(\Omega^1\darrow{0}{df\wedge} Q)$).

Consider the following diagram (commutativity will be checked below in Proposition 4.5) 

\begin{displaymath}
    \xymatrixcolsep{5pc}\xymatrix{
        E_1 \ar[d]_{\varphi_1} \ar[r]^A  & E_0 \ar[d]^{\varphi_0} \ar[r]^{B}  & E_1 \ar[d]^{\varphi_1}  \\
        \voplus{E_1}{\Omega^{1}\otimes E_0} \ar[r]^{\overline{A}} & \voplus{E_0}{\Omega^{1}\otimes E_1} \ar[r]^{\overline{B}} & \voplus{E_1}{\Omega^{1}\otimes E_0} }
\end{displaymath}

$$\text{where} \hp{0.5} \varphi_1 = \begin{bmatrix} 1 \\ 
\nabla_0 A-(1\otimes A)\nabla_1 \end{bmatrix}, \hp{0.3} \varphi_0 =
\begin{bmatrix} 1 \\ \nabla_1B-(1\otimes B)\nabla_0 \end{bmatrix}.$$

\smallskip

We can make the following definition:

\begin{defn}

Define $\varphi_{\mathcal{E},\nabla}: \mathcal{E}\rightarrow \mathcal{E}^{(1)}$ to be the morphism $\begin{bmatrix} 1 \\ At_{\mathcal{E}} \end{bmatrix}$.

\end{defn}

\begin{prop} $\varphi_{\mathcal{E},\nabla}$ is a strict morphism of matrix factorizations. \end{prop}

\begin{proof} First, let's check the commutativity for the square on the left; that is $\overline{A}\circ\varphi_1(x)=\varphi_0\circ A(x)$ for any $x\in E_1$. It's enough to look at the second component. Let $\pi_2$ be the projection to the second component, then

\vp{0.1}

\hp{1.6}$\pi_2\circ \overline{A}\circ\varphi_1(x)=df\wedge
x-B(\nabla_{0}A-A\nabla_{1})(x)$

\vp{0.1}

\hp{2.5}$=df\wedge x-B\nabla_{0}A(x)+f\nabla_{1}(x)$

\vp{0.1}

\hp{2.5}$=df\wedge x+f\nabla_{1}(x)-B\nabla_{0}A(x)$

\vp{0.1}

\hp{2.5}$=\nabla_{1}(f\cdot x)-B\nabla_{0}A(x)$

\vp{0.1}

\hp{2.5}$=(\nabla_{1}B-B\nabla_{0})(A(x))$

\vp{0.1}

\hp{2.5}$=\pi_2\circ\varphi_0\circ A(x)$

The commutativity of the right square can be proved in a similar way. Therefore $\varphi_{\mathcal{E},\nabla}$ is a strict morphism of matrix factorizations. \end{proof}

\begin{prop}

$\varphi_{\mathcal{E},\nabla}$ is independent of the choice of connections up to homotopy.

\end{prop}

\begin{proof}

Suppose we choose other connections for the $E_i$, say $\nabla_i^{'}:E_i\rightarrow \Omega^1_Q\otimes_Q E_i$, $i=0,1$. We show that $\varphi=\varphi_{\mathcal{E},\nabla}$ is homotopic to $\varphi'=\varphi_{\mathcal{E},\nabla'}$.

First, $\nabla_i-\nabla_i'$ is a morphisms of $Q$-modules: for any $q\in Q$, $x\in E_i$, $$(\nabla_i-\nabla_i')(q\cdot x)=\nabla_i(q\cdot x)-\nabla_i'(q\cdot x)=(dq\wedge x+q\cdot\nabla_i(x))-(dq\wedge x+q\cdot\nabla_i'(x))=q\cdot (\nabla_i-\nabla_i')(x).$$

Therefore, we can define $\alpha_0 =
\begin{bmatrix}
   0 \\ (\nabla_{0}-\nabla_0')
\end{bmatrix},
\alpha_1 = \begin{bmatrix}
    0 \\ (\nabla_1-\nabla_1')
\end{bmatrix}$, which live in the following diagram

\begin{displaymath}
    \xymatrixcolsep{14pc}\xymatrix{
        E_1 \ar[d] \ar[r]^A  & E_0 \ar@{.>}[ld]^{\alpha_0} \ar[d]^{\varphi-\varphi'} \ar[r]^{B}  & E_1  \ar@{.>}[ld]^{\alpha_1} \ar[d]  \\
        \voplus{E_1}{\Omega^{1}\otimes E_0} \ar[r]_{\overline{A}} & \voplus{E_0}{\Omega^{1}\otimes E_1} \ar[r]_{\overline{B}} & \voplus{E_1}{\Omega^{1}\otimes E_0} }
\end{displaymath}

It's easy to check that $\overline{A}\circ\alpha_0+\alpha_1\circ B=\varphi-\varphi'$ and similarly for the other square. \end{proof}

Therefore, we usually drop the $\nabla$ from the notation $\varphi_{\mathcal{E},\nabla}$ to simply write it as $\varphi_{\mathcal{E}}$. When $Q$ is local or if we take $E_i$ to be free $Q$-modules, the Atiyah clase is typically like that of the following example.

\begin{ex}

For $\mathcal{E}=(Q^n\xrightarrow{A} Q^n\xrightarrow{B} Q^n), df=AdB+(dA)B$, where $dA=(d_{Q/k}(a_{ij}))$ for the matrix $A=(a_{ij})$. Since $\varphi$ is independent of choice of connection, we use the exterior differential $d$ to construct the Atiyah class, i.e., $\nabla_i=d$ for $i=0,1$. First note that we have $d\circ A-A\circ d=dA\cdot$, because $(d\circ A-A\circ d)(x)=d(A\cdot x)-A\cdot (dx)=dA\cdot x+A\cdot dx-A\cdot dx=dA\cdot x$. Therefore 

$$At_{\mathcal{E}}=\begin{bmatrix} 0 & A \\ B & 0 \end{bmatrix}\cdot \begin{bmatrix} d & 0 \\ 0 & d \end{bmatrix} - \begin{bmatrix} d & 0 \\ 0 & d \end{bmatrix}\cdot \begin{bmatrix} 0 & A \\ B & 0 \end{bmatrix}=\begin{bmatrix} 0 & d\circ A-A\circ d \\ d\circ B-B\circ d & 0 \end{bmatrix}=\begin{bmatrix} 0 & dA \\ dB & 0 \end{bmatrix}$$ hence $At^i_{\mathcal{E}}=$ 

$$\begin{bmatrix} \overbrace{dAdB\cdots dAdB}^i & 0 \\ 0 & \underbrace{dBdA\cdots dBdA}_i \end{bmatrix} \text{($i$ even)   or  } \begin{bmatrix} 0 & \overbrace{dAdB\cdots dA}^i \\ \underbrace{dBdA\cdots dB}_i & 0 \end{bmatrix} \text{($i$ odd)}$$

\end{ex}

\begin{defn}

Define $\mathcal{E}^{(i)}:=(Q \xrightarrow{df\wedge} \Omega^1)^{\otimes i}\otimes_{mf} \mathcal{E}$. Also define morphisms $$\varphi^{(i)}_{\mathcal{E}}:=1^{\otimes i-1}\otimes\varphi_{\mathcal{E}}: \hp{0.2} \mathcal{E}^{(i-1)}\rightarrow \mathcal{E}^{(i)}.$$ By our definition, $\varphi_{\mathcal{E}}=\varphi_{\mathcal{E}}^{(1)}$. Note that $\varphi^{(i)}_{\mathcal{E}}$ can be written in the form: $\begin{bmatrix} I_{2^i} \\ At_{\mathcal{E}} \cdot I_{2^i} \end{bmatrix}$, where $I_{2^i}$ means the $2^i\times 2^i$ identity matrix.

\end{defn}

The following illustrates what we mean by $\mathcal{E}^{(i)}$ and $\varphi^{(i)}_{\mathcal{E}}$:

$$\xymatrix{\mathcal{E}= & E_1 \ar[r]^A \ar[d] & E_0\ar[r]^B \ar[d] & E_1 \ar[d] \\
\mathcal{E}^{(1)}= & \voplus{E_1}{\Omega^1\otimes E_0} \ar[r]^{\overline{A}} \ar[d] & \voplus{E_0}{\Omega^1\otimes E_1} \ar[r]^{\overline{B}} \ar[d] & \voplus{E_1}{\Omega^1\otimes E_0} \ar[d] \\
\mathcal{E}^{(2)}= & \voplus{\voplus{E_1}{\Omega^1\otimes E_0}}{\voplus{\Omega^1\otimes E_0}{\Omega^1\otimes\Omega^1\otimes E_1}} \ar[r]^{\overline{\overline{A}}} \ar[d] & \voplus{\voplus{E_0}{\Omega^1\otimes E_1}}{\voplus{\Omega^1\otimes E_1}{\Omega^1\otimes\Omega^1\otimes E_0}} \ar[r]^{\overline{\overline{B}}} \ar[d] & \voplus{\voplus{E_1}{\Omega^1\otimes E_0}}{\voplus{\Omega^1\otimes E_0}{\Omega^1\otimes\Omega^1\otimes E_1}} \ar[d] \\
\vdots & \vdots & \vdots & \vdots}$$ and for any $e\in\mathcal{E}$, 

$$\varphi_{\mathcal{E}}(e)=\begin{bmatrix} 1 \\ At_{\mathcal{E}} \end{bmatrix}(e)=\begin{bmatrix} e \\ At_{\mathcal{E}} (e)\end{bmatrix}$$

$$\varphi_{\mathcal{E}}^{(1)}(\begin{bmatrix} e \\ At_{\mathcal{E}} (e)\end{bmatrix})=\begin{bmatrix} 1 &  \\  & 1 \\ At_{\mathcal{E}} & \\ & At_{\mathcal{E}}\end{bmatrix} \cdot \begin{bmatrix} e \\ At_{\mathcal{E}} (e)\end{bmatrix}=\begin{bmatrix} e \\ At_{\mathcal{E}}(e) \\ At_{\mathcal{E}}(e) \\ \widetilde{At^2_{\mathcal{E}}} (e)\end{bmatrix}=\begin{bmatrix} e \\ 2 At_{\mathcal{E}}(e) \\  \widetilde{At^2_{\mathcal{E}}}(e) \end{bmatrix}$$ and 

$$\varphi_{\mathcal{E}}^{(2)}(\begin{bmatrix} e \\ 2 At_{\mathcal{E}}(e) \\  \widetilde{At^2_{\mathcal{E}}}(e) \end{bmatrix})=\begin{bmatrix} 1 & & \\ & 1 & \\ & & 1 \\ At_{\mathcal{E}} & & \\ & At_{\mathcal{E}} & \\ & & At_{\mathcal{E}}\end{bmatrix}\cdot \begin{bmatrix} e \\ 2 At_{\mathcal{E}}(e) \\  \widetilde{At^2_{\mathcal{E}}}(e) \end{bmatrix}=\begin{bmatrix} e \\ 2At_{\mathcal{E}}(e) \\  \widetilde{At_{\mathcal{E}}^2}(e) \\ At_{\mathcal{E}}(e) \\ 2\widetilde{At_{\mathcal{E}}^2}(e) \\ \widetilde{At_{\mathcal{E}}^3}(e) \end{bmatrix}=\begin{bmatrix} e \\ 3At_{\mathcal{E}}(e) \\ 3\widetilde{At_{\mathcal{E}}^2}(e) \\ \widetilde{At_{\mathcal{E}}^3}(e)\end{bmatrix}, \hp{0.2} \cdots$$

\begin{cor}

$\varphi^{(i)}_{\mathcal{E}}$ is a strict morphism of matrix factorizations and it is independent of the choice of connections (up to homotopy) for any $i$.

\end{cor}

\subsection{The map $\varphi^n$}\hp{1}

For any nature number $n$,  denote the complex $$\xymatrixcolsep{4pc}\xymatrix{ Q \ar[r]^{ndf\wedge} & \Omega^1
\ar[r]^{(n-1)df\wedge} & \Omega^2 \ar[r]^{(n-2)df\wedge} & \cdots
\ar[r]^{df\wedge} & \Omega^n } \hp{0.2} (*)$$  by $\Omega^{(n)}_{Q,df}$, where $idf\wedge$ denotes left multiplication by $idf$ (i.e., $w_1\wedge\cdots\wedge w_n\mapsto idf\wedge w_1\wedge\cdots\wedge w_n)$, for any $0\leqslant i \leqslant n$. 

There is a natural map of chain complexes $(Q\xrightarrow{df\wedge} \Omega^1)^{\otimes n}\rightarrow \Omega^{(n)}_{Q,df}$, induced from the natural $Q$-module homomorphisms of the following diagram $$\xymatrix{ (Q\oplus\Omega^1)^{\otimes n} \ar@{-->}[rd] \ar@{^{(}->}[r] &  (\oplus_{i \geqslant 0}\Omega^i)^{\otimes n}  \ar[d]^{\wedge} \\ & \oplus_{i\geqslant 0} \Omega^i    }  $$ We again denote this map by $\wedge$. 

\begin{prop}

The map $\wedge: (Q\xrightarrow{df\wedge} \Omega^1)^{\otimes n}\rightarrow \Omega^{(n)}_{Q,df}$ is a map of complexes.

\end{prop}

\begin{proof}

The map $\wedge$ is obviously a $Q$-module homomorphism, so we just need to show $\wedge$ commutes with the differentials of complexes.

Let us denote the differential in $(Q\xrightarrow{df\wedge} \Omega^1)$ by $\partial$ and the differential in $\Omega^{(n)}_{Q,df}$ by $\partial'$. Note that for an element $u\in \Omega^m, \partial' u=(n-m)df\wedge u$ and for $v\in Q\oplus \Omega^1$, $$\partial(v)=\begin{dcases} df\wedge v, \text{ if } |v|=0 \\ 0, \hp{0.35}\text{else} \end{dcases}$$ Therefore, for $a_1\otimes\cdots\otimes a_m\in\Omega^m$, where $a_i\in\Omega^1$, $$\wedge\partial(a_1\otimes\cdots\otimes a_m)=\wedge(\sum_{i=1}^m (-1)^{|a_1|+\cdots+|a_{i-1}|}a_1\otimes\cdots\otimes\partial(a_i)\otimes\cdots\otimes a_m)$$ $$=\sum_{i=1}^m(-1)^{|a_1|+\cdots+|a_{i-1}|}a_1\wedge\cdots\wedge\partial(a_i)\wedge\cdots\wedge a_m$$ $$=\sum_{i=1}^m(-1)^{|a_1|+\cdots+|a_{i-1}|}a_1\wedge\cdots\wedge (df\wedge a_i)\wedge\cdots\wedge a_m$$ $$=\sum_{i=1 \text{ and } |a_i|=0}^m (-1)^{2(|a_1|+\cdots +|a_{i-1}|)} df\wedge a_1\wedge\cdots\wedge a_m$$ $$=(n-m)df\wedge a_1\wedge\cdots\wedge a_m$$ Also, $$\partial'\wedge(a_1\otimes\cdots\otimes a_m)$$ $$=\partial'(a_1\wedge\cdots\wedge a_m)$$ $$=(n-m)df\wedge a_1\wedge\cdots\wedge a_m$$ This completes the proof. \end{proof}

We obviously have: 

\begin{cor}

$\wedge\otimes 1_{\mathcal{E}}: (Q\xrightarrow{df\wedge} \Omega^1)^{\otimes n}\otimes_{mf} \mathcal{E} \rightarrow \Omega^{(n)}_{Q,df}\otimes_{mf} \mathcal{E}$ is a strict morphism of matrix factorizations, for any matrix factorization $\mathcal{E}$.

\end{cor}

\begin{defn}

Define $\varphi_{\mathcal{E}}^n: \mathcal{E}\rightarrow \Omega^{(n)}_{Q,df}\otimes_{mf} \mathcal{E}$ to be the composition $(\wedge\otimes 1_{\mathcal{E}})\circ \varphi_{\mathcal{E}}^{(n)}\circ \varphi_{\mathcal{E}}^{(n-1)} \circ \cdots \circ \varphi_{\mathcal{E}}^{(1)}$; i.e., $\varphi_{\mathcal{E}}^n$ is the composition of the following chain of strict morphisms $$\mathcal{E}\xrightarrow{\varphi^{(1)}} \mathcal{E}^{(1)}\xrightarrow{\varphi^{(2)}} \mathcal{E}^{(2)}\xrightarrow{\varphi^{(3)}} \cdots \xrightarrow{\varphi_{\mathcal{E}}^{(n)}} \mathcal{E}^{(n)}=(Q\xrightarrow{df\wedge} \Omega^1)^{\otimes n}\otimes_{mf} \mathcal{E} \xrightarrow{\wedge\otimes 1_{\mathcal{E}}} \Omega^{(n)}_{Q,df}\otimes_{mf} \mathcal{E}$$

\end{defn}

\begin{cor}

$\varphi^n_{\mathcal{E}}$ is a strict morphism of matrix factorizations and is independent of choice of connections up to homotopy.

\end{cor}

\begin{proof}

We know that each of the $\varphi^{(i)}_{\mathcal{E}}$'s is independent of choice of connections, and thus $\varphi^n_{\mathcal{E}}$ is too. \end{proof}

\begin{prop}

We have $(\wedge\otimes 1_{\mathcal{E}})\circ (1_{\Omega^1[1]^{\otimes(i-1)}}\otimes At)=(\wedge\otimes 1_{\mathcal{E}})\circ (1_{\Omega^{i-1}}\otimes At)\circ(\wedge\otimes 1_{\mathcal{E}})$, for all $i$; that is, the following diagram commutes:

$$\xymatrixcolsep{5pc}\xymatrix{
        \overbrace{\Omega^1[1]\otimes_{mf}\cdots\otimes_{mf}\Omega^1[1]}^{i-1} \otimes_{mf} \mathcal{E} \ar[d]_{1_{\Omega^1[1]^{\otimes(i-1)}}\otimes At} 
        \ar[r]^{\wedge\otimes 1_\mathcal{E}}  & \Omega^{i-1}[i-1]\otimes_{mf} \mathcal{E} \ar[d]^{1_{\Omega^{i-1}} \otimes At}   \\
       \underbrace{\Omega^1[1]\otimes_{mf}\Omega^1[1]\otimes_{mf}\cdots\otimes_{mf}\Omega^1[1]}_i \otimes_{mf} \mathcal{E} 
       \ar[rd]_{\wedge\otimes 1_{\mathcal{E}}} & \Omega^{i-1}[i-1]\otimes_{mf}\Omega^1[1]\otimes_{mf} \mathcal{E} \ar[d]^{\wedge\otimes 1_{\mathcal{E}}}\\ 
         & \Omega^i[i]\otimes_{mf} \mathcal{E} }$$

\end{prop}

\begin{proof}

We can check this directly. For example, it is obvious for $i=1$. \end{proof}

For the sake of simplicity, we will drop all the $1\otimes\cdots\otimes 1$ if there is no confusion from now on. For example, for the above proposition, we will in fact write it as
$(\wedge\otimes 1_{\mathcal{E}})\circ At=(\wedge\otimes 1_{\mathcal{E}})\circ At\circ(\wedge\otimes 1_{\mathcal{E}})$.

Similarly, we have the following corollary.

\begin{cor}

$(\wedge\otimes 1_{\mathcal{E}})\circ \varphi^{(i)}=(\wedge\otimes 1_{\mathcal{E}})\circ
\varphi^{(i)}\circ(\wedge\otimes 1_{\mathcal{E}})$, for any $i$. That is, we
have the following commutative diagram:

$$\xymatrixcolsep{5pc}\xymatrix{
\mathcal{E}^{(i-1)} \ar[d]_{\wedge\otimes 1_{\mathcal{E}}} \ar[r]^{\varphi^{(i)}} & \mathcal{E}^{(i)} \ar[r]^{\wedge\otimes 1_{\mathcal{E}}} & \Omega^{(i)}_{Q,df} \otimes_{mf} \mathcal{E}\\
\Omega^{(i-1)}_{Q,df}\otimes_{mf} \mathcal{E} \ar[r]_{\varphi^{(i)}} &
\Omega^{(i-1)}_{Q,df}\otimes_{mf}(Q\rightarrow \Omega^1)\otimes_{mf} \mathcal{E}
\ar[ur]_{\wedge\otimes 1_{\mathcal{E}}}}$$

\end{cor}

\begin{proof}

Note that \hp{0.2}$(\wedge\otimes 1_{\mathcal{E}})\circ\varphi^{(i)}=(\wedge\otimes 1_{\mathcal{E}})\circ\begin{bmatrix} I_{2^i} \\
At\cdot I_{2^i} \end{bmatrix}$

\vp{0.2}

\hp{2.2} $=\begin{bmatrix} I_{2^i} \\
(\wedge\otimes 1_{\mathcal{E}})\circ At\cdot I_{2^i} \end{bmatrix}$

\vp{0.2}

\hp{2.2} $=\begin{bmatrix} I_{2^i} \\
(\wedge\otimes 1_{\mathcal{E}})\circ At\circ(\wedge\otimes 1_{\mathcal{E}})\cdot I_{2^i}
\end{bmatrix}$

\vp{0.2}

\hp{2.2} $=(\wedge\otimes 1_{\mathcal{E}})\circ \begin{bmatrix} I_{2^i} \\
 At\cdot I_{2^i}
\end{bmatrix}\circ(\wedge\otimes 1_{\mathcal{E}})$

\vp{0.2}

\hp{2.2} $=(\wedge\otimes 1_{\mathcal{E}})\circ \varphi^{(i)}\circ(\wedge\otimes
1_{\mathcal{E}}).$ \end{proof}

\begin{cor}

We have $\varphi^n=\displaystyle\sum\limits_{i=0}^n\dbinom{n}{i}At^i\in \Omega^{(n)}_{Q,df}\otimes_{mf}\mathcal{E}$.

\end{cor}

\begin{proof} The base case is clear, now by induction, say $\varphi^{n-1}=\displaystyle\sum\limits_{i=0}^{n-1}\dbinom{n-1}{i}At^i$,
then

\vp{0.2}

\hp{2}$\varphi^n=(\wedge\otimes
1_{\mathcal{E}})\circ\varphi^{(n)}\circ\varphi^{(n-1)}\circ\cdots\circ\varphi$

\vp{0.2}

\hp{2.15} $=(\wedge\otimes 1_{\mathcal{E}})\circ
\varphi^{(n)}\circ(\wedge\otimes
1_{\mathcal{E}})\circ\varphi^{(n-1)}\circ\cdots\circ\varphi$

\vp{0.2}

\hp{2.15} $=(\wedge\otimes 1_{\mathcal{E}})\circ
\varphi^{(n)}\circ\varphi^{n-1}$

\vp{0.2}

\hp{2.15} $=(\wedge\otimes 1_{\mathcal{E}})\circ
\varphi^{(n)}\circ\displaystyle\sum\limits_{i=0}^{n-1}\dbinom{n-1}{i}At^i$

\vp{0.2}

\hp{2.15} $=\begin{bmatrix} I_{2^n} \\
 (\wedge\otimes 1_{\mathcal{E}})\circ At\cdot I_{2^n}
\end{bmatrix}\circ\displaystyle\sum\limits_{i=0}^{n-1}\dbinom{n-1}{i}At^i$

\vp{0.2}

\hp{2.15} $=\displaystyle\sum\limits_{i=0}^{n-1}\dbinom{n-1}{i}At^i+
At(\displaystyle\sum\limits_{i=0}^{n-1}\dbinom{n-1}{i}At^i)$

\vp{0.2}

\hp{2.15} $=\displaystyle\sum\limits_{i=0}^n\dbinom{n}{i}At^i.$ \end{proof}

\subsection{The map $\widetilde{\varphi^n}$}\hp{1}

For any nature number $n$, define $(\Omega^\bullet, df, n)$ to be the complex: $$\xymatrixcolsep{4pc}\xymatrix{ Q \ar[r]^{df\wedge} & \Omega^1 \ar[r]^{df\wedge} & \Omega^2 \ar[r]^{df\wedge} & \cdots \ar[r]^{df\wedge} & \Omega^n }.$$

From now on, we assume in addition that $k\supset \mathbb{Q}$. Under this assumption, there is an isomorphism of complexes $\Omega^{(n)}_{Q,df}\rightarrow (\Omega^\bullet, df,n)$ defined by

$$\xymatrixcolsep{4pc}\xymatrix{ Q \ar[r]^{ndf\wedge} \ar[d]^= & \Omega^1
\ar[r]^{(n-1)df\wedge} \ar[d]^{\frac{1}{n}} & \Omega^2 \ar[r]^{(n-2)df\wedge}
\ar[d]^{\frac{1}{n(n-1)}} & \cdots \ar[r]^{df\wedge} &
\Omega^n \ar[d]^{\frac{1}{n!}} \\
Q \ar[r]^{df\wedge} & \Omega^1 \ar[r]^{df\wedge} & \Omega^2 \ar[r]^{df\wedge} & \cdots
\ar[r]^{df\wedge} & \Omega^n}$$

\vp{0.2}

We compose $\varphi^n$ with the above isomorphism to get a morphism which now called $\widetilde{\varphi^n}$; i.e., $$\widetilde{\varphi^n}: \xymatrixcolsep{2pc}\xymatrix{ \mathcal{E} \ar[r] &
(\Omega^\bullet, df,n)\otimes_{mf} \mathcal{E}}$$

By Proposition 4.16, we have an expression for $\widetilde{\varphi^n}$:

$$\widetilde{\varphi^n}=\displaystyle\sum\limits_{i=0}^n\frac{1}{i!}At^i$$


\section {Chern character for matrix factorizations and its basic
properties}

\vp{0.2}

For a $k$-algebra $Q$, where $k$ is a commutative unital ring that contains $\mathbb{Q}$, we constructed (for any given $n\in \mathbb{N}$) a strict morphism of matrix factorizations $\widetilde{\varphi^n}: \mathcal{E}\rightarrow (\Omega^\bullet,df,n)\otimes_{mf} \mathcal{E}$ in last section. Now, we can define a Chern character for matrix
factorizations. 

\subsection{Supertrace}\hp{1}

\vp{0.1}

Let $Q$ be any commutative ring and $M$ a finitely generated projective $Q$-module. Let $M^*=Hom_Q(M,Q)$ be the dual of $M$. Consider the two maps $$ \xymatrixcolsep{2pc}\xymatrix{ End_Q(M) & M^*\otimes_Q M
\ar[l]_{\xi} \ar[r]^{\epsilon} & Q }$$ given by $\xi(\alpha\otimes n)(m)=\alpha(m)n$, with $m,n\in M, \alpha\in M^*$ and by $\epsilon(\alpha\otimes n)=\alpha(n)$ respectively. If $M$ is a finitely generated projective $Q$-module, then $\xi$ is an isomorphism and the composite $\epsilon\circ\xi^{-1}$ is the standard trace map: $\epsilon\circ\xi^{-1}=tr.$ Suppose that $M$ is free of finite rank over $Q$, with $Q$ a $k$-algebra. Then a $Q$-linear map $M\rightarrow M\otimes_Q\Omega^*_{Q/k}$, upon choice of basis, is a matrix with coefficients in $\Omega^*_{Q/k}$, that is, an element of $End_Q(M)\otimes_Q\Omega^*_{Q/k}$. Since a projective module is a direct summand of a free module, the same is true when $M$ is projective, i.e., $$Hom_Q(M, M\otimes_Q\Omega^*)\cong End_Q(M)\otimes_Q \Omega^*.$$ Therefore, when $\mathcal{M}$ is a matrix factorization, the underlying module $M$ is projective so $At^i$ can be viewed as an element of $End_Q(M)\otimes_Q\Omega^i_{Q/k}$.

\begin{defn}

Given a $\mathbb{Z}/2$-graded finitely generated projective $Q$-module $M$ and an endomorphism $T$ of $M$ of degree $0$, using that $End_Q(M)_0=End_Q(M_0\oplus M_1)_0=End_Q(M_0)\oplus End_Q(M_1)$, define the \emph{supertrace} to be $$str(T):=tr(T_0)-tr(T_1)\in Q$$ where $T=T_0\oplus T_1$ with $T_i\in End_Q(M_i), i=0,1$.

\end{defn} 

\begin{prop} \hp{1}

\begin{enumerate}

\item If $\alpha, \beta: \mathcal{E}\rightarrow\mathcal{E}$ are strict morphisms of matrix factorization and $\alpha$ is homotopic to $\beta$, then $str(\alpha)=str(\beta)$.

\item $str$ is an invariant under cyclic permutations, i.e., $$str(\alpha_1\circ \cdots \circ \alpha_n)=str(\alpha_{\sigma(1)}\circ\cdots\circ \alpha_{\sigma(n)})$$ for $\sigma$ a cyclic permutation of $n$ elements.

\end{enumerate}

\end{prop}

\begin{proof}

Say $\mathcal{E}=(E_1\darrow{A}{B} E_0)$ \begin{enumerate}

\item There are $Q$-module homomorphisms $x: E_0\rightarrow E_1$ and $y: E_1\rightarrow E_0$ such that $Ax+yB=\alpha_0-\beta_0$ and $By+xA=\alpha_1-\beta_1$. So $$(Ax+yB)-(By+xA)=(\alpha_0-\beta_0)-(\alpha_1-\beta_1)$$ $$(Ax-xA)-(yB-By)=(\alpha_0-\alpha_1)-(\beta_0-\beta_1)$$ $$tr(Ax-xA)-tr(yB-By)=tr(\alpha_0-\alpha_1)-tr(\beta_0-\beta_1)$$ $$(tr(Ax)-tr(xA))-(tr(yB)-tr(By))=0=str(\alpha)-str(\beta)$$ $$str(\alpha)=str(\beta).$$

\item This is obvious since $tr$ is an invariant under cyclic permutations. 

\end{enumerate}

\end{proof}

\subsection{Chern character}\hp{1}

\vp{0.1}

Before making the definition of the Chern character of matrix factorizations, let's first recall the definition of a smooth algebra and prove a few propositions necessary for the definition.

\begin{defn}

\begin{enumerate}

\item If $k$ is an algebraically closed field, then $Q$ is \emph{smooth of relative dimension $d$} if it is of finite type, its dimension is $d$, and the module $\Omega^1_{Q/k}$ of differentials is a finitely generated locally free $Q$-module of rank $d$.

\item Let $k$ be an arbitrary field, $\overline{k}$ its algebraic closure. Then $Q$ is \emph{smooth of relative dimension $d$} if $Q\otimes_k\overline{k}$ is smooth of relative dimension $d$ over $\overline{k}$.

\item Let $\theta: Q\rightarrow Q'$ be a ring map, then $\theta$ is \emph{smooth of relative dimension $d$} if it is flat, finitely presented, and for all primes $\fp$ of $Q$, the fibre ring $k(\fp)\otimes_Q Q'$ is smooth of relative dimension $d$ over $k(\fp)$, where $k(\fp)$ is the residue field at $\fp$.

\end{enumerate}

\end{defn}

\begin{prop}

Suppose $Q$ is a smooth $k$-algebra of relative dimension $d$ where $k$ is a commutative unital ring that contains $\mathbb{Q}$. Then $str(\widetilde{\varphi^d})=str(\widetilde{\varphi^{d+1}})=\cdots$.

\end{prop}

\begin{proof}

We have $str(At^{d+1})\in\Omega^{d+1}\otimes_{mf} \cE=0$. Therefore $$str(\widetilde{\varphi^{d+1}})=\displaystyle\sum\limits_{i=0}^{d+1}\frac{1}{i!}str(At^i)=1+str(At)+\cdots+\frac{1}{d!}str(At^d)+\frac{1}{(d+1)!}str(At^{d+1})$$

\hp{2.27} $=1+str(At)+\cdots+\dfrac{1}{d!}str(At^d)+0$

\vp{0.1}

\hp{2.27} $=str(\widetilde{\varphi^d})$

i.e., $str(\widetilde{\varphi^d})=str(\widetilde{\varphi^{d+1}})$ and hence $str(\widetilde{\varphi^d})=str(\widetilde{\varphi^{d+i}})$ for any $i\geqslant 1$. \end{proof}

\begin{prop}

Given any matrix factorization $\cE=(E_1\xrightarrow{A} E_0\xrightarrow{B} E_1)\in MF(Q,f), df\wedge str(At^i_{\cE})=0$ in $\Omega^{i+1}$ for any $i$. If $i$ is an odd integer, $str(At^i_{\cE})=0$.

\end{prop}

\begin{proof}

For the underlying finitely generated projective $Q$-module $E=E_0\oplus E_1$ of $\cE$, the trace homomorphism $tr$ is $End(E)\otimes \Omega^{\bullet}\rightarrow\Omega^{\bullet}$, so $str(At^i_{\cE})\in \Omega^i$ since it's the differentce of two elements in $\Omega^i$.

It's enough to check this locally, so we adopt the notations used in Example 4.7. In particular, since the map $\Omega^i_{Q/k}\rightarrow \Omega^{i+1}_{Q[\frac{1}{f}]/k}$ is injective, it suffices to check $df\wedge str(At^i_{\cE})=0$ after inverting $f$. Therefore we may assume $A$ and $B$ are invertible matrices with entries in $Q[\frac{1}{f}]$. Notice that $str(At^i_{\cE})=0$ when $i$ is odd so we just need to show this when $i$ is an even integer. 

Since $A$ is invertible, $B=f\cdot A^{-1}$, therefore $$dB=df\cdot A^{-1}+f\cdot dA^{-1}$$ since $A^{-1}\cdot A=I$ we have $dA^{-1}=-A^{-1}dA\cdot A^{-1}$ hence $$dB=df\cdot A^{-1}-fA^{-1}dA\cdot A^{-1}$$

Also, since $df=dA\cdot B+A\cdot dB$ and say $i=2l$, $$df\wedge str(At^i_{\cE})=2tr(df\overbrace{dAdBdAdB\cdots dAdB}^i)$$ $$=2tr(dfdA(df\cdot A^{-1}-fA^{-1}dA\cdot A^{-1})\cdots dA(df\cdot A^{-1}-fA^{-1}dA\cdot A^{-1}))$$ $$=2tr(dfdA(-fA^{-1}dA\cdot A^{-1})\cdots dA(-fA^{-1}dA\cdot A^{-1}))$$ $$=(-1)^l2tr(\underbrace{dA\cdot A^{-1}\cdots dA\cdot A^{-1}}_{\text{even numbers}})=0$$ 

We have the last equality because switching matrices of odd forms is only going to introduces a sign. Also, the product $(dA\cdot A^{-1})^{\text{even}}=(dA\cdot A^{-1})^{\text{odd}}\cdots(dA\cdot A^{-1})^{\text{odd}}$ stays the same no matter how you switch, therefore it has to be zero. \end{proof}

By the above proposition, we know that $str(At^i_{\cE})$ vanishes when $i$ is odd and is a cycle (of the complex $(\Omega^{\bullet}, df, n)$) when $i$ is even, so it defines an element of the homology. 

Now we are ready to give our definition of the Chern character. Assume that $Q$ is now a smooth $k$-algebra of relative dimension $n$ with $k$ a commutative unital ring that contains $\mathbb{Q}$.

\begin{defn}

We define the {\em Chern character} of $\cE\in MF(Q,f)$ to be $$ch(\cE):=str(\widetilde{\varphi^n_{\cE}})=\displaystyle\sum\limits_{i=0}^n\frac{1}{i!}str(At^i_{\cE})\in H_0((\Omega^{\bullet}, df, n)_{\mathbb{Z}/2})=\bigoplus\limits_{i=0}^n \dfrac{ker(\Omega^{2i}\xrightarrow{df\wedge}\Omega^{2i+1})}{Im(\Omega^{2i-1}\xrightarrow{df\wedge}\Omega^{2i})}.$$

\end{defn}

Recall that $(\Omega^{\bullet}, df, n)$ is the complex $\xymatrixcolsep{3pc}\xymatrix{ Q \ar[r]^{df\wedge} & \Omega^1 \ar[r]^{df\wedge} & \Omega^2 \ar[r]^{df\wedge} & \cdots \ar[r]^{df\wedge} & \Omega^n }$ and $(\Omega^{\bullet}, df, n)_{\mathbb{Z}/2}$ is the $\mathbb{Z}/2$-folding of this complex. Also, notice that by Proposition 5.5, $str(At^i_|{\cE})=0$ when $i$ is an odd integer. Therefore the Chern character is in fact $ch(\cE)=\sum\limits_{i\geqslant 0} \dfrac{1}{(2i)!} str(At^{2i}_{\cE})$.

Our definition of the Chern is the same as the one in Platt \cite{platt2012chern}. For the special case when $Q=k[[x_1,\cdots, x_n]]$ and $f\in Q$ an isolated singularity, see the following example.

\begin{ex}

Let $f\in Q=k[[x_1,\cdots,x_n]]$ be an isolated singularity at the origin (that is, the localizations at every prime except the maximal ideal $\fm=(x_1,\cdots, x_n)$ is regular), $E=(Q^r \darrow{A}{B} Q^r)$ a matrix factorization with $d=\begin{bmatrix} 0 & A \\B & 0 \end{bmatrix}$. $E$ is a free $Q$-module, so as before, we can choose the exterior differential $d$ to be the connections and we get that $At_{\cE}=\begin{bmatrix} 0 & dA\\ dB & 0 \end{bmatrix}$. Also, notice that in this situation $(\Omega^{\bullet}, df, n)$ is exact except in position $n$ (the dimension of $Q$) \cite{walker2013support}. Therefore we have that $ch(\cE)=\displaystyle\sum\limits_{i=0}^n\frac{1}{i!}str(At^i_{\cE})=\frac{1}{n!}str(dAdB\cdots dAdB)$, which agrees with the Chern character obtained by \cite{kapustin2003topological}, \cite{platt2012chern} and \cite{segal2009closed}. It differs by a sign with the ones in \cite{carqueville2012adjunctions}, \cite{dyckerhoff2010kapustin}, \cite{dyckerhoff2011pushing} and \cite{polishchuk2010chern}. 

\end{ex}

\begin{prop}

Given matrix factoriztions $\cE=(E_1\xrightarrow{A} E_0\xrightarrow{B} E_1)$ and $\cE'=(E'_1\xrightarrow{C} E'_0\xrightarrow{S} E'_1)$in MF(Q,f), a strict morphism $\b: \cE' \rightarrow \cE$, then the following diagram commute up to homotopy $$\xymatrixcolsep{5pc}\xymatrix{ \cE' \ar[d]^{\b} \ar[r]^{\widetilde{\varphi_{\cE}}} & (Q\xra{df\wedge}\Omega^1)\otimes_{mf} \cE'\ar[d]^{1\otimes\beta}\\
\cE \ar[r]^{\widetilde{\varphi_{\cE'}}} & (Q\xra{df\wedge}\Omega^1)\otimes_{mf} \cE }$$

\end{prop}

\begin{proof}

Recall that $\widetilde{\varphi_{\cE}}=\begin{bmatrix} 1 \\ At_{\cE} \end{bmatrix}=1+At_{\cE}$.

Choose connections $\nabla_i$ and $\nabla_i'$, then we can construct module homomorphisms $\psi_0 = \begin{bmatrix}
   0 \\ (\nabla_{0}\beta-(1\otimes\beta)\nabla_{0}^{'})
\end{bmatrix},
 \psi_1 = \begin{bmatrix}
    0 \\ (\nabla_{1}\beta-(1\otimes\beta)\nabla_{1}^{'})
\end{bmatrix}$, which lives in the diagram

\begin{displaymath}
    \xymatrixcolsep{14pc}\xymatrix{
        E_1^{'} \ar[d] \ar[r]^C  & E_0^{'} \ar@{.>}[ld]_{\psi_0} \ar[d]^{\widetilde{\varphi_{\cE}}\circ\beta-(1\otimes\beta)\circ
\widetilde{\varphi_{\cE'}}} \ar[r]^{D}  & E_1^{'}  \ar@{.>}[ld]^{\psi_1} \ar[d]  \\
        \voplus{E_1}{\Omega^{1}\otimes E_0} \ar[r]_{\overline{A}} & \voplus{E_0}{\Omega^{1}\otimes E_1} \ar[r]_{\overline{B}} & \voplus{E_1}{\Omega^{1}\otimes E_0} }
\end{displaymath}

where $\widetilde{\varphi_{\cE}}\circ\beta-(1\otimes\beta)\circ
\widetilde{\varphi_{\cE'}}$
is the matrix $\begin{bmatrix} 0 \\
At_{\cE}\circ\beta-(1\otimes\beta)\circ At_{\cE'}
\end{bmatrix}$.

\vp{0.2}

First of all $\psi_0$ and $\psi_1$ are indeed module homomorphism:

\vp{0.1}

\hp{0.95} $\psi_0(q\cdot x)=\nabla_{0}\beta(q\cdot
x)-(1\otimes\beta)\nabla_{0}^{'}(q\cdot x)$

\vp{0.1}

\hp{1.5} $=\nabla_{0}(q\cdot\beta(x))-(1\otimes\beta)(dq\wedge
x+q\cdot\nabla_{0}^{'}(x))$

\vp{0.1}

\hp{1.5}
$=dq\wedge\beta(x)+q\cdot\nabla_{0}\beta(x)-dq\wedge\beta(x)-(1\otimes\beta)(q\cdot\nabla_{0}^{'}(x))$

\vp{0.1}

\hp{1.5}
$=q\cdot\nabla_{0}\beta(x)-q\cdot(1\otimes\beta)(\nabla_{0}^{'}(x))$

\vp{0.1}

\hp{1.5} $=q(\nabla_{0}\beta-(1\otimes\beta)\nabla_{0}^{'})(x)$

\vp{0.1}

\hp{1.5} $=q\cdot\psi_0(x)$

\vp{0.1}

The same argument shows that $\psi_1$ is also a module homomorphism.

\vp{0.2}

We want to show that $\psi_0$ and $\psi_1$ give us a homotopy. For the degree $0$ part, we need to show $(\widetilde{\varphi_{\cE}}\circ\beta-(1\otimes\beta)\circ
\widetilde{\varphi_{\cE'}})_0=\overline{A}\circ \psi_0+\psi_1\circ D$.

\vp{0.2}

Recall that $\overline{A} = \begin{bmatrix} A & 0 \\
df\wedge & -B \end{bmatrix} $ and elements are $2\times 1$ column vectors, the equality in the first row is easy so we just need to check the equality for the second row. Hence,

\hp{1.8} $(At_{\cE}\circ\beta-(1\otimes\beta)\circ At_{cE'})_0$

\vp{0.1}

\hp{1.7}
$=(\nabla_{1}B-B\nabla_{0})\beta-(1\otimes\beta)(\nabla_{1}^{'}D-D\nabla_{0}^{'})$

\vp{0.1}

\hp{1.7}
$=\nabla_{1}B\beta-B\nabla_{0}\beta-(1\otimes\beta)\nabla_{1}^{'}D+(1\otimes\beta)D\nabla_{0}^{'}$

\vp{0.1}

\hp{1.7} $=\nabla_{1}\beta
D-B\nabla_{0}\beta-(1\otimes\beta)\nabla_{1}^{'}D+B(1\otimes\beta)\nabla_{0}^{'}$

\vp{0.1}

\hp{1.7}
$=(\nabla_{1}\beta-(1\otimes\beta)\nabla_{1}^{'})D-B(\nabla_{0}\beta-(1\otimes\beta)\nabla_{0}^{'})$

\vp{0.1}

\hp{1.7} $=\psi_1\circ D+\overline{A}\circ \psi_0$

In the above calculation, we use $B\beta=\beta D$ and
$(1\otimes\beta)D=B(1\otimes\beta)$ by the fact that $\beta$ is a
strict morphism of matrix factorizations, i.e., the following
commutative diagram:

\begin{displaymath}
    \xymatrix{
        E_1^{'} \ar[d]_\beta \ar[r]^C  & E_0^{'} \ar[d]_\beta \ar[r]^D  & E_1^{'} \ar[d]_\beta  \\
        E_1 \ar[r]^A & E_0\ar[r]^B & E_1 }
\end{displaymath} 

\end{proof}

\begin{cor}

Under the same hypothesis as in Proposition 5.8, we have $\widetilde{\varphi_{\cE}^n}\circ\beta\sim(1\otimes1\otimes\cdots\otimes 1\otimes\beta)\circ\widetilde{\varphi_{\cE'}^{n}}$.

\end{cor}

\begin{proof}

Indeed, $$\widetilde{\varphi_{\cE}^n}\circ\beta=\widetilde{\varphi_{\cE}^{n-1}}\circ\widetilde{\varphi_{\cE}}\circ\beta\sim \widetilde{\varphi_{\cE}^{n-1}}\circ(1\otimes\beta)\circ\widetilde{\varphi_{\cE'}}$$ then by induction $$\widetilde{\varphi_{\cE}^n}\circ\beta\sim(1\otimes1\otimes\cdots\otimes 1\otimes\beta)\circ\widetilde{\varphi_{\cE'}^{n}}.$$

\end{proof}

\begin{cor}

We have $ch(\cE)=ch(\cE')$ for homotopy equivalent matrix factorizations $\cE$ and $\cE'$ of $MF(Q,f)$.

\end{cor}

\begin{proof}

Say we have $\cE\darrow{\alpha}{\b} \cE' $, such that $\alpha\circ\beta\sim 1_{\cE'}$ and $\beta\circ\alpha\sim 1_{\cE}$, by Corollary 5.9 $$\widetilde{\varphi_{\cE}^n}\circ\beta\circ\a\sim(1\otimes 1\otimes\cdots\otimes 1\otimes\beta)\circ\widetilde{\varphi_{\cE'}^{n}}\circ\a$$

Therefore, by Proposition 5.2

\hp{2.5} $str(\widetilde{\varphi_{\cE}^n})$

\vp{0.1}

\hp{2.4} =$str(\widetilde{\varphi_{\cE}^n}\circ\beta\circ\alpha)$

\vp{0.1}

\hp{2.4} =$str((1^{\otimes n}\otimes\beta)\circ\widetilde{\varphi_{\cE'}^n}\circ\alpha)$

\vp{0.1}

\hp{2.4} =$str(\alpha\circ(1^{\otimes n}\otimes\beta)\circ\widetilde{\varphi_{\cE'}^n})$ 

\vp{0.1}

\hp{2.4} =$str(\widetilde{\varphi_{\cE'}^n})$

\vp{0.1}

This gives $$ch(\cE)=ch(\cE').$$ 

\end{proof}

\begin{thm}

Given any distinguished triangle $\cP \xrightarrow{\theta} \cQ \rightarrow cone(\theta) \rightarrow \cP[1]$ in $[MF(Q,f)]$, we have $$ch(\cQ)=ch(\cP)+ch(cone(\theta)) \hp{1}  (*)$$

\end{thm}

\begin{proof}

We will prove this theorem by explicit calculation of the Chern character.

First, it's enough to check equality for the even components, as discussed in the proof of Proposition 5.5. By definition $ch(\cE)=str(\widetilde{\varphi^n_{\cE}})=str(\displaystyle\sum\limits_{i=0}^n\frac{1}{i!}At_{\cE}^i)=\displaystyle\sum\limits_{i=0}^n\frac{1}{i!}str(At_{\cE}^i)$ for any matrix factorization $\cE$, so it's enough to prove$str(At_{\cQ}^{2i})=str(At_{\cP}^{2i})+str(At_{cone(\theta)}^{2i})$, for all even integers $2i$ between $1$ and $n$.

Say $\cP=(P_1\xrightarrow{A} P_0\xrightarrow{B} P_1)$ and $\cQ=(Q_1\xrightarrow{C} Q_0\xrightarrow{D} Q_1)$, the mapping cone is $$cone(\theta)=(\xymatrixcolsep{5pc}\xymatrix{Q_1\oplus P_0 \ar@<1ex>[r]^{ \begin{bmatrix}
    C & \theta_0 \\ 0 & -B
\end{bmatrix}} & Q_0\oplus P_1 \ar@<1ex>[l]^{\begin{bmatrix}
   D & \theta_1 \\ 0 & -A
\end{bmatrix}} })$$

Choose any connections $\nabla_0^{\cP}$ and $\nabla_1^{\cP}$ for $\cP$, similarly $\nabla_0^{\cQ}$ and $\nabla_1^{\cQ}$ for $\cQ$. We have induced connections for $cone(\theta)$:

$$\nabla_1^{cone(\theta)} = \begin{bmatrix}
   \nabla_0^{\cQ} &  \\  & \nabla_1^{\cP}
\end{bmatrix}  \hp{0.7}
 \nabla_0^{cone(\theta)} = \begin{bmatrix}
    \nabla_1^{\cQ} &  \\  & \nabla_0^{\cP}
\end{bmatrix}$$

\vp{0.2}

Since the Chern character is independent of choice of connections, we use these to compute the Atiyah class $At_{cone(\theta)}$ for $cone(\theta)$, which is just

$$At_{cone(\theta)}=\begin{bmatrix} \nabla_1^{\cQ} & & & \\  & \nabla_0^{\cP} & & \\ & & \nabla_0^{\cQ} &  \\ & & & \nabla_1^{\cP} \end{bmatrix}\cdot \begin{bmatrix}  & & C & \theta_0 \\  &  & & -B \\ D & \theta_1 & &  \\ & -A & & \end{bmatrix}-\begin{bmatrix}  & & C & \theta_0 \\  &  & & -B \\ D & \theta_1 & &  \\ & -A & & \end{bmatrix}\cdot \begin{bmatrix} \nabla_1^{\cQ} & & & \\  & \nabla_0^{\cP} & & \\ & & \nabla_0^{\cQ} &  \\ & & & \nabla_1^{\cP} \end{bmatrix}$$

$$=\begin{bmatrix}  & & X & * \\  &  & & Z \\ Y & * & &  \\ & W & & \end{bmatrix}$$

where  $$X=(\nabla_1^{\cQ} C-C \nabla_0^{\cQ})$$ $$Y=(\nabla_0^{\cQ} D-D\nabla_1^{\cQ})$$ $$Z=(B \nabla_1^{\cP} - \nabla_0^{\cP} B)$$ $$W=(A \nabla_0^{\cP} - \nabla_1^{\cP} A).$$

Hence $$At_{cone(\theta)}^2=\begin{bmatrix} XY & * & & \\  &  ZW & & \\  & & YX & * \\  & & & WZ \end{bmatrix}$$

Therefore, $$At_{cone(\theta)}^{2i}=\begin{bmatrix} (XY)^i & * & & \\  &  (ZW)^i & & \\  & & (YX)^i & * \\  & & & (WZ)^i \end{bmatrix}$$ for any even integer $2i$ between $1$ and $n$.

This gives that

\vp{0.2}

\hp{2.1} $str(At_{cone(\theta)}^{2i})$

\vp{0.2}

\hp{2} $=tr\begin{bmatrix} (XY)^i & *  \\  &
 (ZW)^i
\end{bmatrix}-tr\begin{bmatrix} (YX)^i & * \\   &
 (WZ)^i\end{bmatrix}$

\vp{0.2}

\hp{2.05}$=2tr((XY)^i) - 2tr((WZ)^i).$

\vp{0.2}

hence

$$str(At_{\cP}^{2i})+str(At_{cone(\theta)}^{2i})= 2tr((WZ)^i)+2tr((XY)^i)-2tr((WZ)^i)$$

\hp{2.6}$=2tr((XY)^i)=str(At_{\cQ}^{2i}).$ \end{proof}

\subsection{Grothendieck group}\hp{1}

\vp{0.1}

Recall that the Grothendieck group $K_0(T)$ of a triangulated category $T$ is the free abelian group generated by isomorphism classes for objects of $T$, modulo the relations $[X]+[Z]=[Y]$ for distinguished triangles $X\rightarrow Y\rightarrow Z\rightarrow X[1]$.

\begin{cor}

The Chern character induces a map from $K_0([MF(Q,f)])$ to $H_0((\Omega^{\cdot},df,n)_{\mathbb{Z}/2})$.

\end{cor}

\begin{proof}

Any distinguished triangle is isomorphic (in the homotopy category) to a triangle of the form of Theorem 5.11. Now apply Corollary 5.10 and Theorem 5.11 \end{proof}

Now we will prove that the Chern character is a ring homomorphism.

\begin{lem} $\bigoplus\limits_{f\in Q}K_0([MF(Q,f)])$ is a ring via $[\cE]_f\cdot[\cF]_g:=[\cE\otimes_{mf} \cF]_{f+g}$.

\end{lem}

\begin{proof}

First we have to show that the above multiplication is well-defined.

Since we know from the definition that $-\otimes_{mf}-$ preserves homotopy equivalences of matrix factorizations. For any given $\cE\simeq \cE'$ and $\cF\simeq \cF'$, we have $\cE\otimes_{mf} \cF\simeq \cE'\otimes_{mf} \cF'$, so the tensor product is well-defined on the free abelian group generated by isomorphism classes of matrix factorizations; we denote this group by $\bigoplus\limits_{f\in Q} \mathbb{Z}([MF(Q,f)])$.

\vp{0.1}

Now, let's show that $\bigoplus\limits_{f\in Q} \mathbb{Z}([MF(Q,f)])$ is a commutative ring under the above multiplication.

\begin{enumerate}

\item $(\cE)_f\cdot(\cF)_g\in \bigoplus\limits_{f\in Q} \mathbb{Z}([MF(Q,f)])$

\item $\cE\otimes_{mf} \cF\cong \cF\otimes_{mf}\cE$ hence $(\cE)_f\cdot(\cF)_g=(\cF)_g\cdot (\cE)_f$

\item $((\cE)_f\cdot(\cF)_g)\cdot(\cG)_h=(\cE\otimes_{mf} \cF)_{f+g}\cdot(\cG)_h=((\cE\otimes_{mf} \cF)\otimes_{mf} \cG)_{f+g+h}$. 

Also, $(\cE)_f\cdot((\cF)_g\cdot(\cG)_h)=(\cE)_f\cdot((\cF\otimes_{mf} \cG)_{g+h})=(\cE\otimes_{mf}(\cF\otimes_{mf} \cG))_{f+g+h}$, where $(\cE)_f$ means the isomorphism class of the matrix factorization $\cE\in MF(Q,f)$. Hence the above shows that the multiplication is associative.

\vp{0.1}

\item There is an identity  $\textbf{1}=(0 \darrow{0}{0} Q)\in MF(Q,0)$ such that $ (\textbf{1})_0\cdot(\cE)_f=(\cE)_f$.

\vp{0.1}

Indeed, $\textbf{1}\otimes_{mf}\cE$ equals to $$0\oplus E_1 \darrow{\begin{bmatrix}
    & 1\otimes e_1 \\ -1\otimes e_0 &
\end{bmatrix}}{\begin{bmatrix}
      & -1\otimes e_1 \\ 1\otimes e_0 &
\end{bmatrix}}
 0\oplus E_0.$$

Therefore we have an isomorphism of $\textbf{1}\otimes_{mf} \cE$ and $\cE$, i.e., $ (\textbf{1})_0\cdot(\cE)_f=(\cE)_f$.

Similarly, we also have $(\cE)_f\cdot(\textbf{1})_0=(\cE)_f$.

\item The last thing to check is the distribution law. In particular, I will show that $$((\cE)_f+(\cF)_f)\cdot (\cG)_g=(\cE)_f\cdot (\cG)_g+(\cF)_f\cdot (\cG)_f$$ for any $\cE,\cF\in MF(Q,f)$ and $\cG\in MF(Q,g)$.

In fact, we have

\hp{1.3}$((\cE)_f+(\cF)_f)\cdot (\cG)_g=(\cE\oplus \cF)_f\cdot (\cG)_g$

\hp{2.28}$=((\cE\oplus \cF)\otimes_{mf} \cG)_{f+g}$

\hp{2.28}$=((\cE\otimes_{mf} G)\oplus(\cF\otimes_{mf} \cG))_{f+g}$

\hp{2.28}$=(\cE\otimes_{mf} \cG)_{f+g}+(\cF\otimes_{mf} \cG)_{f+g}$

\hp{2.28}$=(\cE)_f\cdot(\cG)_g+(\cF)_f\cdot(\cG)_g$

\vp{0.1}

Similarly, we can prove that $(\cG)_g\cdot((\cE)_f+(\cF)_f)=(\cG)_g\cdot (\cE)_f + (\cG)_g\cdot (\cF)_f$.

The above shows that the isomorphism classes of all matrix factorizations is a monoid under $-\otimes_{mf}-$, so $\bigoplus\limits_{f\in Q} \mathbb{Z}([MF(Q,f)])$ is in fact a commutative ring. 

Finally, to show that this multiplication is well-defined on the quotient group, it's enough to prove that the subgroup $$\{[\cQ]-[\cP]-[\cW]: \cP\rightarrow \cQ\rightarrow \cW\rightarrow \cP[1] \text{ a distinguished triangle}\}$$ is an ideal of $\bigoplus\limits_{f\in Q} \mathbb{Z}([MF(Q,f)])$. This amounts to the following fact: tensor product is a triangulated functor \cite{yu2013geometric}.

\end{enumerate}

\end{proof}

\begin{defn}: Define $K_f(Q):=\bigoplus\limits_{i\in \mathbb{Z}_{\geqslant 0}}K_0([MF(Q,if)])$; this is in fact a  subring of $\bigoplus\limits_{f\in Q}K_0([MF(Q,if)])$.

\end{defn}

\begin{proof}

Given any $[a],[b]\in K_f(Q)$, $[a]+[b]\in K_f(Q)$; $[a]\cdot [b]\in K_f(Q)$; $[\textbf{1}]\in K_f(Q)$; $[-a] \in K_f(Q)$ (since $[-a]=[a[1]]\in K_f(Q).$) \end{proof}

\begin{lem}

$\bigoplus\limits_{f\in Q}H_0((\Omega^{\cdot},df,n)_{\mathbb{Z}/2})$ is a commutative ring via  $\wedge$.

\end{lem}

\begin{proof}

First, assume $n$ is an odd integer. A similar proof works when $n$ is even. Recall that $(\Omega^{\cdot},df,n)$ is the following complex: $$\xymatrixcolsep{3pc}\xymatrix{ Q \ar[r]^{df\wedge} & \Omega^1 \ar[r]^{df\wedge} & \Omega^2 \ar[r]^{df\wedge} & \cdots \ar[r]^{df\wedge} & \Omega^n }.$$ Therefore the $\mathbb{Z}/2$-folding $(\Omega^{\cdot},df,n)_{\mathbb{Z}/2}$ is the matrix factorization $\cE=(E_1\darrow{D_1}{D_0} E_0)$ where $$E_1=\Omega^1\oplus\Omega^3\oplus\cdots\oplus\Omega^n, \hp{0.5} E_0=Q\oplus\Omega^2\oplus\cdots\oplus\Omega^{n-1}$$ $$\text{and} \hp{0.5} D_1=\begin{bmatrix} 0 & 0 & \cdots & 0 \\ df\wedge & 0 & \cdots & 0 \\ 0 & df\wedge & \cdots & 0 \\ & \cdots & \cdots & \\ 0 & 0 & \cdots & df\wedge \end{bmatrix}_{\frac{n+1}{2}\times\frac{n+1}{2}} \hp{0.5} D_0=\begin{bmatrix} df\wedge & 0 & \cdots & 0 \\ 0 & df\wedge & \cdots & 0 \\ & \cdots & \cdots & \\ 0 & 0 & \cdots & df\wedge \end{bmatrix}_{\frac{n+1}{2}\times\frac{n+1}{2}}$$

Hence $H_0((\Omega^{\cdot},df,n)_{\mathbb{Z}/2})=\dfrac{ker D_0}{Im D_1}$ is a ring by properties of the wedge product. For example, for any $a,b,c\in\bigoplus\limits_{i\text{ even}}\Omega^i$, we have $$(a\wedge b)\wedge c=a\wedge (b\wedge c).$$ It is not hard to see that elements on the two sides of the above equation determine the same element in homology. The same holds for other conditions to make a set into a ring. It is commutative since we are dealing only with even former (in general $a\wedge b=(-1)^{ij} b\wedge a$ for $a\in\Omega^i$ and $b\in\Omega^j$; therefore, $i,j$ even means $a\wedge b=b\wedge a$). From this we see that $\bigoplus\limits_{f\in Q}H_0((\Omega^{\cdot},df,n)_{\mathbb{Z}/2})$ is a commutative ring.

It is clear that $\bigoplus\limits_{i\in\mathbb{Z}_{\geqslant 0}}H_0((\Omega^{\cdot},idf,n)_{\mathbb{Z}/2})$ is a subring of $\bigoplus\limits_{f\in Q}H_0((\Omega^{\cdot},df,n)_{\mathbb{Z}/2})$. \end{proof}

\begin{thm}

Given two matrix factorizations $\cE\in [MF(Q,f)]$ and $\cF\in [MF(Q,g)]$, we have the following commutative diagram 

$$\xymatrixcolsep{10pc}\xymatrix{ \cE\otimes_{mf} \cF \ar[rd]_{\widetilde{\varphi_{\cE\otimes_{mf}\cF}^n}}
\ar[r]^{\widetilde{\varphi_{\cE}^n}\otimes\widetilde{\varphi_{\cF}^n}} &
((\Omega^{\bullet},df,n)\otimes_{mf} \cE)\otimes_{mf}((\Omega^{\bullet},dg,n)\otimes_{mf}
\cF) \ar[d]^{\wedge\circ(1\otimes\tau\otimes 1}\\
& (\Omega^{\bullet},df+dg,n)\otimes_{mf}(\cE\otimes_{mf} \cF)}$$ where $\tau: \cE\otimes_{mf}(\Omega^{\bullet},df,n)\rightarrow(\Omega^{\bullet},df,n)\otimes_{mf} \cE$ is the isomorphism $\tau(a\otimes b)=(-1)^{|a||b|}b\otimes a$.

\end{thm}

\begin{rem} The above diagram makes sense, since the relative dimension of $Q$ over $k$ is $n$. After changing the position, wedging things together, terms with degree higher than $n$ vanish. \end{rem}

\begin{proof} (of Theorem5.16) First, by Proposition 3.5, for the underlying modules $E$ and $F$, if we choose a connection $\nabla_E$ for $E$ and $\nabla_F$ for $F$, then there is a natural connection for the tensor product: $\nabla_E\otimes 1 + 1\otimes\nabla_F$. Also, the differential for the tensor product of two matrix factorizations is given by $d_{\cE\otimes_{mf} \cF}(e\otimes f)=d_{\cE}(e)\otimes f+(-1)^{|e|}e\otimes d_{\cF}(f)$, where $e\in E$ and $f\in F$.  After a careful calculation, we have that $$At_{\cE\otimes_{mf} \cF}(e\otimes f)=At_{\cE}(e)\otimes f+(-1)^{|e|} e\otimes At_{\cF}(f)  \hp{0.3} (*)$$ i.e., $At_{\cE\otimes_{mf} \cF}=At_{\cE}\otimes 1+\tau(1\otimes At_{\cF})$, where $\tau$ is the map in the statement of the theorem.

Another observation we want to make before looking into $\widetilde{\varphi_{\cE\otimes_{mf} \cF}^n}$ is that $$\wedge\circ(At_{\cE}\otimes 1)\circ\tau(1\otimes At_{\cF})=\wedge\circ\tau(1\otimes At_{\cF})\circ (At_{\cE}\otimes 1).$$ In fact, we have $$\wedge\circ(At_{\cE}\otimes 1)\circ\tau(1\otimes At_{\cF})(e\otimes f)=\wedge\circ(At_{\cE}\otimes 1)\circ\tau(1\otimes At_{\cF})((-1)^{|e|}\sigma (e\otimes At_{\cF}(f)))$$ where $\sigma$ is the same as $\tau$ but doesn't introduce a sign. Say for simplicity that $At_{\cF}(f)=u\otimes f'$ and $At_{\cE}(e)=w\otimes e'$ (these should really be sums of simple tensors, nonetheless, the idea is the same and the case for simple tensors is more clear), then the above is $$\wedge\circ(At_{\cE}\otimes 1)\circ((-1)^{|e|}u\otimes e\otimes f')$$ $$=(-1)^{|e|}\wedge(u\otimes At_{\cE}(e)\otimes f')$$ $$=(-1)^{|e|}\wedge(u\otimes w\otimes e'\otimes f')$$ $$=(-1)^{|e|}\cdot(u\wedge w\otimes e'\otimes f')$$ $$=-(-1)^{|e|}\cdot(w\wedge u\otimes e'\otimes f')$$ $$=-(-1)^{|e|}\wedge\sigma(At_{\cE}(e)\otimes At_{\cF}(f)).$$ For $\wedge\circ\tau(1\otimes At_{\cF})\circ (At_{\cE}\otimes 1)(e\otimes f)$, we have $$\wedge\circ\tau(1\otimes At_{\cF})\circ (At_{\cE}\otimes 1)(e\otimes f)$$ $$=\wedge\circ\tau(1\otimes At_{\cF})\circ (At_{\cE}(e)\otimes f)$$ $$=\wedge\circ\tau(1\otimes At_{\cF})\circ (w\otimes e' \otimes f)$$ $$=\wedge\circ\tau(w\otimes e' \otimes At_{\cF}(f))$$ $$=\wedge\circ\tau(w\otimes e' \otimes u\otimes f')$$ $$=(-1)^{|e'|}\wedge(w\otimes u \otimes e'\otimes f')$$ $$=(-1)^{|e|+1}(w\wedge u \otimes e'\otimes f')$$ $$=-(-1)^{|e|}\wedge\sigma(At_{}\cE(e)\otimes At_{\cF}(f)).$$

Therefore, the tow compositions of the operators $At_{\cE}\otimes 1$ and $\tau(1\otimes At_{\cF})$ are the same after $\wedge$ and more importantly we get $-(-1)^{|e|}\wedge\sigma(At_{\cE}(e)\otimes At_{\cF}(f))$ applying to the element $e\otimes f$. Then it is not hard to see that $$\wedge(At_{\cE}\otimes 1)^k\circ(\tau()1\otimes At_{\cF})^s=   \begin{dcases*}
        \wedge\sigma(At_{\cE}^k(e)\otimes At_{\cF}^s(f))  & if one of $k,s$ is even\\
        -(-1)^{|e|}\wedge\sigma(At_{\cE}^k(e)\otimes At_{\cF}^s(f)) & if both $k$ and $s$ are odd 
        \end{dcases*}$$

We can compute $\widetilde{\varphi_{\cE\otimes_{mf} \cF}^n}$ by formula $(*)$, the degree $i$th piece is (remember the notation $\sim$ indicates we have already applied $\wedge$ to the Atiyah class) $$\frac{1}{i!}At_{\cE\otimes_{mf}\cF}^i=\frac{1}{i!}(At_{\cE}\otimes 1+\tau(1\otimes At_{\cF}))^i$$ 

$$=\begin{dcases*}
        \frac{1}{i!}\displaystyle\sum_{k+s=i}\dbinom{i}{k}\wedge\sigma(At_{\cE}^k\otimes
At_{\cF}^s)  & if one of $k,s$ is even\\
        -(-1)^{|e|}\frac{1}{i!}\displaystyle\sum_{k+s=i}\dbinom{i}{k}\wedge\sigma(At_{\cE}^k\otimes
At_{\cF}^s) & if both $k$ and $s$ are odd 
        \end{dcases*}$$   

$$=\begin{dcases*}
        \displaystyle\sum_{k+s=i}\frac{1}{k!s!}\wedge\sigma(At_{\cE}^k\otimes
At_{\cF}^s)  & if one of $k,s$ is even\\
        -(-1)^{|e|}\displaystyle\sum_{k+s=i}\frac{1}{k!s!}\wedge\sigma(At_{\cE}^k\otimes
At_{\cF}^s) & if both $k$ and $s$ are odd 
        \end{dcases*}$$ 
       
Meanwhile, the $i$th component for $\widetilde{\varphi_{\cE}^n}\otimes\widetilde{\varphi_{\cF}^n}$ is $\displaystyle\sum_{k+s=i}\frac{1}{k!s!}(At_{\cE}^k\otimes At_{\cF}^s)$. Therefore, say $At_{\cE}^k(e)=w'\otimes\overline{e}$ and $At_{\cE}^s(f)=u'\otimes\overline{f}$ with $w'\in\Omega^k$ and $u'\in\Omega^s$ (hence $|\overline{e}|=|e|+1$ if $k$ is odd and $|\overline{e})|=|e|$ if $k$ is even), we have $$(\wedge\circ(1\otimes\tau\otimes 1))(\displaystyle\sum_{k+s=i}\frac{1}{k!s!}(At_{\cE}^k(e)\otimes At_{\cF}^s(f)))$$
$$=(-1)^{|\overline{e}|\cdot s}\cdot(\displaystyle\sum_{k+s=i}\frac{1}{k!s!}\wedge\sigma(At_{\cE}^k(e)\otimes At_{\cF}^s(f))$$ 

$$=\begin{dcases*}
        -(-1)^{|e|}\cdot\displaystyle\sum_{k+s=i}\frac{1}{k!s!}\wedge\sigma(At_{\cE}^k(e)\otimes At_{\cF}^s(f))  & if both $k,s$ are odd\\
        \displaystyle\sum_{k+s=i}\frac{1}{k!s!}\cdot\wedge\sigma(At_{\cE}^k(e)\otimes At_{\cF}^s(f)) & otherwise 
        \end{dcases*}$$ 

This completes the proof of the theorem. \end{proof}

\begin{cor}

The Chern character $ch: K_f(Q)\rightarrow \bigoplus\limits_{i\in  \mathbb{Z}_{\geqslant 0}}H_0((\Omega^{\cdot},idf,n)_{\mathbb{Z}/2})$ is a ring homomorphism, i.e., $$ch([\cE]\cdot[\cF])=ch([\cE])ch([\cF])$$

\end{cor}

\begin{proof} Theorem 5.16 tells us that $\widetilde{\varphi_{\cE\otimes_{mf}\cF}^n}=(\wedge\circ(1\otimes \tau\otimes 1))\circ(\widetilde{\varphi_{\cE}^n}\otimes\widetilde{\varphi_{\cF}^n})$. The corollary follows by applying $str$ to this equation. \end{proof}

\vp{0.5}

\subsection{Functoriality}\hp{1}

\vp{0.1}

Consider a $k$-algebra homomorphism $\varphi: R\rightarrow S$ that sends $f\in R$ to $g\in S$. For any matrix factorization $\cE=(E_1\darrow{A}{B} E_0)\in MF(R,f)$, there is then a naturally induced matrix factorization $$\cE\otimes_{mf} S=(E_1\otimes_R S \darrow{A\otimes 1}{B\otimes 1} E_0\otimes_R S)\in MF(S,g).$$ It is obvious that $E_1\otimes_R S$ and $E_0\otimes_R S$ are finitely generated projective $S$-modules. Also, we do have $(B\otimes 1)\circ(A\otimes 1)=(A\otimes 1)\circ(B\otimes 1)=g\cdot id$: in fact, $(B\otimes 1)\circ(A\otimes 1)(e_1\otimes s)=(f\cdot e_1)\otimes s$, but since we are talking about $S$-modules, $(f\cdot e_1)\otimes s=e_1\otimes (\varphi(f)\cdot s)=e_1\otimes (g\cdot s)=g\cdot(e_1\otimes s)$. 

\begin{defn} 

For a $k$-algebra homomorphism $\varphi$ as above, define a functor $\varphi_*: MF(R,f)\rightarrow MF(S,g)$ that sends $\cE$ to $\varphi_*(\cE):=\cE\otimes_{mf} S$ and a strict morphism $\a=(\a_0,\a_1):\cE\rightarrow\cF$ to a strict morphism $\varphi_*(\a):=(\a_0\otimes 1,\a_1\otimes 1): \varphi_*(\cE)\rightarrow\varphi_*(\cF)$.

\end{defn} 

The functor $\varphi_*$ is well-defined on the homotopy category of matrix factorizations. Also we can talk about $\varphi_*(E)=E\otimes_R S$ for a finitely generated $R$-module $E$ by regarding $E$ as a matrix factorization of zero. In particular, there is a natural map $\mu: \Omega^1_{R/k}\otimes_R S=\varphi_*(\Omega^1_{R/k})\rightarrow\Omega^1_{S/k}$ which sends $d_{R/k}(r)\otimes s$ to $s\cdot d_{S/k}(\varphi(r))$. 

Let us prove a lemma before getting into the statement about functoriality.

\begin{lem} 
For $\varphi$ as above and any finitely generated projective $R$-module $E$, there is a naturally induced connection $\nabla_{E\otimes_R S}:=\mu(\nabla_E\otimes 1)+\sigma(1\otimes d_{S/k})$ for the $S$-module $\varphi_*(E)$, i.e., $$\nabla_{E\otimes_R S}: \varphi_*(E):=E\otimes_R S\rightarrow \Omega^1_{S/k}\otimes_S (E\otimes_R S)\cong \Omega^1_{S/k}\otimes_R E$$ 
\end{lem}

\begin{proof}

First, notice that we have the following two compositions: $$E\otimes_k S\xrightarrow{\nabla_E\otimes 1}(\Omega^1_{R/k}\otimes_R E)\otimes_k S\cong \Omega^1_{R/k}\otimes_R (E\otimes_k S)\cong (\Omega^1_{R/k}\otimes_R S)\otimes_k E\xrightarrow{\mu}\Omega^1_{S/k}\otimes_k E$$ and $$E\otimes_k S\xrightarrow{1\otimes d_{S/k}} E\otimes_R\Omega^1_{S/k}\xrightarrow{\sigma}\Omega^1_{S/k}\otimes_k E$$

Let's denote the some of the above two compositions by $\nabla_{E\otimes_k S}$. It is obvious that they are both $k$-linear, one can also show that $\nabla_{E\otimes_k S}$ is in fact $R$-linear by checking directly. Hence we get an induced map: $$\xymatrix{E\otimes_k S \ar[r]^{\nabla_{E\otimes_k S}} \ar@{->>}[d] & \Omega^1_{S/k}\otimes_k E \ar@{->>}[r]  & \Omega^1_{S/k}\otimes_R E\\  E\otimes_R S \ar@{-->}[ru]  & & }$$ which we denote by $\nabla_{E\otimes_R S}$.

Now the only thing left to check is that $\nabla_{E\otimes_R S}$ satisfies the Leibniz rule, i.e., $$\nabla_{E\otimes_R S}(s\cdot(e\otimes s'))=d_{S/k}(s)\otimes(e\otimes s')+s\cdot\nabla_{E\otimes_R S}(e\otimes s'),$$ for any $e\in E, s,s'\in S$.

Let's prove it using the same technique as in the proof of Theorem 5.16. Say $\nabla_E(e)=d_{R/k}(r)\otimes e'\in\Omega^1_{R/k}\otimes_R E$, $$\nabla_{E\otimes_R S}(s\cdot(e\otimes s'))=\nabla_{E\otimes_R S}(e\otimes ss')$$ $$=\mu(\nabla_E(e)\otimes ss')+\sigma(e\otimes d_{S/k}(ss'))$$ $$=\mu(d_{R/k}(r)\otimes e'\otimes ss')+(d_{S/k}(s)\cdot s'+s\cdot d_{S/k}(s'))\otimes e$$ $$=ss'\cdot d_{S/k}(\varphi(r))\otimes e'+(d_{S/k}(s)\cdot s'+s\cdot d_{S/k}(s'))\otimes e.$$ 

Meanwhile, $$d_{S/k}(s)\otimes(e\otimes s')+s\cdot\nabla_{E\otimes_R S}(e\otimes s')=s'\cdot d_{S/k}(s)\otimes e+s\cdot(\mu(\nabla_E\otimes 1)(e\otimes s')+\sigma(1\otimes d_{S/k})(e\otimes s'))$$ $$=s'\cdot d_{S/k}(s)\otimes e+s\cdot(\mu(\nabla_E(e)\otimes s')+d_{S/k}(s')\otimes e)$$ $$=s'\cdot d_{S/k}(s)\otimes e+s\cdot d_{S/k}(s')\otimes e+s\cdot\mu(d_{R/k}(r)\otimes e'\otimes s')$$ $$=s'\cdot d_{S/k}(s)\otimes e+s\cdot d_{S/k}(s')\otimes e+s\cdot(s'd_{S/k}(\varphi(r))\otimes e').$$ \end{proof}

For simplicity (and to make future calculations easier), we follow the usual convention of denoting $\nabla_{E\otimes_R S}$ by $\nabla_E\otimes 1+1\otimes d_{S/k}$. Now we can state and prove

\begin{prop}[Functoriality]

Under the above hypotheses and the extra condition that both $R$ and $S$ are smooth $k$-algebras with the same relative dimension $n$, we have $\varphi_*\circ ch=ch\circ\varphi_*$.

\end{prop}

\begin{proof}

By our formula for the Chern character, it's enough to show $\varphi_*(At(\cE))=At(\varphi_*(\cE))$ for a matrix factorization $\cE$.

By Lemma 5.20, choose the nature connection $\nabla_E\otimes 1+1\otimes d_{S/k}$ for $E\otimes_R S$, so the Atiyah class of $\varphi_*(\cE)=\cE\otimes_{mf} S$ is 

$$\begin{bmatrix} \nabla_E\otimes 1+1\otimes d_{S/k} & \\ & \nabla_E\otimes 1+1\otimes d_{S/k}\end{bmatrix}\begin{bmatrix} & A\otimes 1\\ B\otimes 1 & \end{bmatrix} - $$ 

$$\begin{bmatrix} & A\otimes 1\\ B\otimes 1 & \end{bmatrix}\begin{bmatrix} \nabla_E\otimes 1+1\otimes d_{S/k} & \\ & \nabla_E\otimes 1+1\otimes d_{S/k}\end{bmatrix}$$

$$=\begin{bmatrix} & (\nabla_E A-A\nabla_E)\otimes 1 \\ (\nabla_E B-B\nabla_E)\otimes 1 & \end{bmatrix}$$ The Atiyah class of $\cE$ is $$\begin{bmatrix} & \nabla_E A-A\nabla_E \\ \nabla_E B-B\nabla_E & \end{bmatrix}$$ Now is obvious from the definition of $\varphi_*$ on strict morphisms that $\varphi_*(At(\cE))=At(\varphi_*(\cE))$.\end{proof}

\section*{Acknowledgements}
This is part of the author's PhD thesis written under the supervision of Professor Mark Walker at the University of Nebraska-Lincoln. I would like to thank him for his help and advice of this article. 

\bibliographystyle{plain}
\bibliography{Refs}

\vp{0.2}

\emph{\footnotesize{Department of Mathematics, University of Nebraska-Lincoln, Lincoln NE 68588 USA}} 

\emph{\footnotesize{Current address: Departement Wiskunde-Informatica, Universiteit Antwerpen, Middelheimlaan 1, 2020 Antwerpen, Belgium}}

\emph{\footnotesize{E-mail address: xuan.yu@uantwerpen.be}}

\end{document}